\documentclass[11pt]{article}%
\usepackage{geometry}
\usepackage{amsmath}
\usepackage{amsfonts}
\usepackage{amssymb}
\usepackage{graphicx}%
\providecommand{\U}[1]{\protect\rule{.1in}{.1in}}
\newtheorem{theorem}{Theorem}
\newtheorem{theorem1}{Theorem}

\newtheorem{corollary}[theorem]{Corollary}

\newtheorem{definition}[theorem]{Definition}
\newtheorem{example}[theorem]{Example}

\newtheorem{lemma}[theorem]{Lemma}
\newtheorem{notation}[theorem]{Notation}

\newtheorem{proposition}[theorem]{Proposition}
\newtheorem{remark}[theorem]{Remark}

\newenvironment{proof}[1][Proof]{\noindent\textbf{#1.} }{\ $\square$}
\newenvironment{proof2}[1][]{\noindent }{\ $\square$}

\begin{document}

\title{An occupation time formula for semimartingales in $\mathbb{R}^{N}$}
\author{A. Bevilacqua, F. Flandoli}
\maketitle

\begin{abstract}
Inspired by coarea formula in geometric measure theory, an occupation time
formula for continuous semimartingales in $\mathbb{R}^{N}$ is proven. The occupation
measure of a semimartingale, for $N\geq2$, is singular with respect to
Lebesgue measure but it has a bounded density ``transversal'' to a foliation,
under proper assumptions. In the particular case of the foliation given
locally by the distance function from a manifold, the transversal
density is related to a geometric local time of the semimartingale at the
manifolds of the foliation.

\end{abstract}

\section{Introduction}
The occupation time formula for real valued continuous semimartingales reads%
\begin{equation*}
\int_{0}^{T}f\left(  X_{t}\right)  d\left\langle X\right\rangle _{t}%
=\int_{\mathbb{R}}f\left(  a\right)  L_{T}^{a}\left(  X\right)  da
\end{equation*}
for all positive Borel functions $f:\mathbb{R\rightarrow R}$, where $L_{T}%
^{a}\left(  X\right)  $ is the local time at $a$ of $X$ over $\left[
0,T\right]  $. This formula is related to the correction term of It\^{o}
formula and explains its extensions beyond $C^{2}$ functions. The aim of this
paper is to discuss a possible extension to semimartingales in $\mathbb{R}%
^{N}$, $N\geq2$.

Let $X$ be a continuous semimartingale in $\mathbb R^N$. When $\langle X^j,X^k\rangle_t$ is differentiable a.s$.$ in $(t,\omega)$ for all $j,k\in 1,\dots,N$, we introduce the matrix valued process $g_t$ defined as $g_t^{jk}=d\langle X^j,X^k\rangle_t/dt$. In \textbf{Section \ref{3}} we prove a multidimensional extension of occupation time formula inspired by coarea formula in geometric measure theory. A particular but relevant case of Theorem \ref{Thm occupation time formula localized} is the following statement.

\begin{theorem}\label{general occup}
Assume that $cI_n\le g_t\le CI_N$ a.s$.$ in $(t,\omega)$ for some constants $C>c>0$. Let $\phi\in C^2(\mathbb R^N)$ be a function such that $\inf_{x\in A} |\nabla\phi(x)|>0$ on an open set $A\subseteq\mathbb R^N$. Then there exists a random bounded compact support non-negative function $\mathcal{L}_{i,A,\phi}^{a}$ and random probability measures $Q_{A}^{i}\left(a,dx\right)$, $Q_{A,\phi}^{i}\left(  a,\cdot\right)  $ concentrated on $\Gamma_{a}=\left\{  x\in A:\phi\left(  x\right)  =a\right\}  $ for a.e$.$ $a\in\mathbb{R}$, such that%
\begin{align}\label{occupation time formula introduction}
\int_{0}^{T}1_{X_{t}\in A}f\left(  X_{t}\right)  d\left\langle X^{i}\right\rangle _{t}=\int_{\mathbb{R}}\left(  \int_{\Gamma_{a}}f\left(  x\right)  Q_{A,\phi}^{i}\left(  a,dx\right)  \right)  \mathcal{L}_{i,A,\phi}^{a}da.
\end{align}
\end{theorem}

We omit to denote the dependence of $\mathcal{L}_{i,A,\phi}^{a}$ also on $T$ and $X$ since they
will always be a priori given. The formula extends to $\int_{0}^{T}f\left(X_{t}\right)  d\left\langle X^{i},X^{j}\right\rangle _{t}$ by polarization.

Localization on a set $A$ is important in applications, since the non-degeneracy conditions may be too severe on the full space. In section \ref{2_2} we give more general conditions of non-degeneracy of $X$ and $\phi$ than those assumed in Theorem \ref{general occup}, and accordingly, Theorem \ref{Thm occupation time formula localized} is more general; this additional generality has applications to certain singular problems, as described in Sections \ref{Sect applications} and \ref{Sect embed}.

As we said, formula (\ref{occupation time formula introduction}) has the
flavor of coarea formula in geometric measure theory. However, the classical
coarea formula disintegrates Lebesgue measure along a foliation and the
measures $Q\left(  a,dx\right)  $ on the leaves are the $\left(  N-1\right)
$-dimensional Hausdorff measures. Here we disintegrate a random measure
$\mu_{i}$, defined by
\[
\mu_{i}\left(  f\right)  =\int_{0}^{T}f\left(  X_{t}\right)  d\left\langle
X^{i}\right\rangle _{t}%
\]
which is singular with respect to Lebesgue measure, for $N\geq2$. Thus the
measures $Q^{i}_{A,\phi}\left(  a,dx\right)  $ do not have good regularity properties.

The key point is, on the contrary, that the measure in the ``transversal''
direction\ to the leaves, $\mathcal{L}_{i,A,\phi}^{a}da$, has a density with respect
to Lebesgue measure. This fact is false if we replace $X$ by a smooth
function: the occupation measure of a smooth deterministic function may
concentrate at some point, and the transversal density $\mathcal{L}^{a}$ is
lost. Thus it is the \textit{regularity of fluctuations} of a semimartingale
with good non-trivial quadratic variation that produces the densities
$\mathcal{L}_{i,A,\phi}^{a}$, similarly to the existence of the local time in
dimension 1. Formula (\ref{occupation time formula introduction}) captures a
regularity property of the occupation measures $\mu_{i}$ of certain
semimartingales. Other regularity properties of occupation meausure are
reviewed for instance in \cite{Flandoli}. The regularity of
occupation measure is also a topic of interest in harmonic analysis, see for
instance \cite{TaoWright}.

By formula (\ref{occupation time formula introduction}), the integral
$\int_{0}^{T}\left\vert f\left(  X_{t}\right)  \right\vert d\left\langle
X^{i}\right\rangle _{t}$ is finite for functions $f$ that are singular along
an $\left(  N-1\right)  $-dimensional manifold $\Gamma$, with a certain degree
of integrability of the singularity. In \textbf{Section \ref{Sect applications}}, devoted to examples and applications, we call this class of functions $L^1_{loc}(\Gamma^\bot)$ and we prove a general result of integrability.

\begin{theorem}\label{integ}
Let $X$ be a continuous semimartingale in $\mathbb{R}^{N}$ such that $cI_n\le g_t\le CI_N$ a.s$.$ in $(t,\omega)$ for some constants $C>c>0$. Let $\Gamma$ be an $(N-1)$-dimensional orientable manifold of class $C^{2}$, closed and without boundary, embedded in $\mathbb R^N$. If $f\in L^1_{loc}(\Gamma^\bot)$ then $P\left(  \int_{0}^{T}\left\vert f\left(  X_{t}\right)  \right\vert dt<\infty\right)  =1$.
\end{theorem}

This result is, for some applications, more precise than the results offered by other approaches, in which the integrability degree of $f$ is related to the dimension $N$, see Remark \ref{R integ}.

A natural question is whether $\mathcal{L}_{i,A,\phi}^{a}$ is equal to the local time of some
1-dimensional semimartingale. When this holds true, the process $a\mapsto\mathcal{L}_{i,A,\phi}%
^{a}$ has a c\`{a}dl\`{a}g modification. Another question is whether
$\mathcal{L}_{i,A,\phi}^{a}$ is a sort of local time of $X$ at the leave
$\Gamma_{a}$. We answer these two questions in a very particular case: when $\phi$ is locally the distance from a given manifold $\Gamma$ and $X$ is a
Brownian semimartingale, namely when $X$ is continuous and $g=I_N$. In \textbf{Section \ref{Sect embed}}, denoted the identical
$\mathcal{L}_{i,A,\phi}^{a}$ by $\mathcal{L}_{A,\phi}^{a}$, we prove its local representation as the (symmetric) local time of an 1-dimensional semimartingale.

\begin{theorem}
Let $X$ be a Brownian semimartingale in $\mathbb{R}^{N}$. Let $\Gamma$ be an $(N-1)$-dimensional orientable manifold of class $C^{2}$, closed and without boundary, embedded in $\mathbb R^N$. Then, for the distance function $d\left(  \cdot,\Gamma\right)  $, there
exists a neighborhood $\mathcal{V}$ of $\Gamma$ and an extension $\phi$
outside $\mathcal{V}$ such that the assumptions of Theorem \ref{Thm occupation time formula localized} are satisfied
with $A=\mathcal{V}$, and additionally, we have that for each fixed $\omega\in\Omega$ there exists $\varepsilon_1(\omega)>0$ such that%
\begin{equation*}
\mathcal{L}_{A,\phi}^{a}=\tilde L_{T}^{a}\left(\phi\left(X\right)\right)
\end{equation*}
for a.e. $a<\varepsilon_1(\omega)$, and they are both null if $a<0$. In particular, on the random interval $(-\infty,\varepsilon_1]$ the process $(\omega,a)\mapsto \mathcal L^a_{A,\phi}(\omega)$ is the modification of a c\`{a}dl\`{a}g process.
\end{theorem}

In \textbf{Section \ref{2}} we define the random variable $L_{T}^{\Gamma_{a}}\left(  X\right)$ and we call it the
\textit{geometric local time of }$X$\textit{ at }$\Gamma_{a}$\textit{ on}
$\left[  0,T\right]$. In the case $\Gamma$ has an uniform neighborhood in which the distance its regular (see Section \ref{unif}), we prove an additional representation of $\mathcal{L}^{a}_{A,\phi}$ as the geometrical local time at the leave $\Gamma_{a}$.

\begin{theorem}
Under the hypotheses of Theorem \ref{Thm L=L=L}, if there exists $\varepsilon>0$ such that $\mathcal V=\left\{  x\in\mathbb{R}^{N}:d\left(  x,\Gamma\right)  <\varepsilon\right\} $, then we have
\begin{equation*}
\mathcal{L}_{A,\phi}^{a}=L_{T}^{\Gamma_{a}}\left(  X\right)
\end{equation*}
for a.e. $a\in\left[ 0,\varepsilon_{0}\right)  $, where
$\Gamma_{a}=\left\{  x\in\mathcal V:d(x,\Gamma)=a\right\}$.
\end{theorem}

Then we compare $L_{T}^{\Gamma_{a}}\left(  X\right)$ with the similar but different local time on graphs defined by Peskir
\cite{Peskir1},\cite{Peskir}. The research reported here has been strongly influenced by it. For the purpose of a generalized It\^{o} formula for $u\left(X_{t}\right)  $ where $u:\mathbb{R}^{N}\rightarrow\mathbb{R}$ is smooth except on a graph, the notion of \cite{Peskir1},\cite{Peskir} is very convenient. We have tried to apply that notion by local graph charts to our set-up, in order to avoid new definitions, but this turns out to be not easy and thus we prefer to develop the definition of $L_{T}^{\Gamma_{a}}\left(  X\right)$ from scratch in part inspired by \cite{Sato}. The definition given here is conceptually similar to \cite{Peskir} but it has two differences which may be of
interest: i) it treat $\left(  N-1\right)  $-dimensional manifolds which are
not necessarily global graphs, ii)\ it is intrinsic of the manifold and does
not depend on the coordinate system.

\section{The occupation time formula\label{3}}

\subsection{Disintegration of random measures}
Given a finite Borel measure $\mu$ on $\mathbb{R}^{N}$ and a Borel function
$\phi:\mathbb{R}^{N}\rightarrow\mathbb{R}$, set
$$\nu=\phi_{\sharp}\mu$$
the push-forward of $\mu$ under $\phi$ ($\nu=\mu\circ\phi^{-1}$). Then there
exists a family of probability measures $\left(  \mu_{a}\right)
_{a\in\mathbb{R}}$ on $\mathbb{R}^{N}$, uniquely determined for $\nu$-a.e.
$a\in\mathbb{R}$, called conditional probabilities of $\mu$ w.r.t. $\nu$, such that:

\begin{description}
\item[i)] for every Borel set $E$, $a\mapsto\mu_{a}\left(  E\right)  $ is measurable

\item[ii)]
\[
\mu\left(  f\right)  =\int_{\mathbb{R}}\left(  \int_{\Gamma_{a}}f\left(
x\right)  \mu_{a}\left(  dx\right)  \right)  \nu\left(  da\right)
\]
for all positive Borel functions $f:\mathbb{R}^{N}\rightarrow\mathbb{R}$,
where $\Gamma_{a}=\left\{  x\in\mathbb{R}^{N}:\phi\left(  x\right)
=a\right\}  $

\item[iii)] $\mu_{a}\left(  \Gamma_{a}\right)  =1$, for $\nu$.a.e. $a\in\mathbb{R}$.
\end{description}

This is a consequence of Rohlin disintegration theorem (see \cite{Simmons}, \cite{CaraDane} for a recent version and references therein).

Let $\left\{  \mu^{\omega}\text{; a.e. }\omega\in\Omega\right\}  $ be a random
finite Borel measure on $\mathbb{R}^{N}$, on a probability space $\left(
\Omega,\mathcal{F},P\right)$ that is universally measurable (see \cite{Simmons}); for instance, Polish spaces are universally measurable.

\begin{proposition}\label{rohlin}
The push-forward $\nu^{\omega}=\phi_{\sharp}\mu^{\omega}$ is a
random finite Borel measure on $\mathbb{R}$; moreover the family of probability measures $\left\{
\mu_{a}^{\omega}\text{; }\nu\text{-a.e. }a\in\mathbb{R}\text{, }P\text{-a.e.
}\omega\in\Omega\right\}  $ has the properties that for every Borel set $E$,
$\left(  a,\omega\right)  \mapsto\mu_{a}^{\omega}\left(  E\right)  $ is
measurable on $\left(  \mathbb{R}\times\Omega,\mathcal{B}\left(
\mathbb{R}\right)  \otimes\mathcal{F}\right)  $; and for $P$-a.e. $\omega
\in\Omega$, properties 2 and 3 above hold true.
\end{proposition}

\begin{proof}
The proof that $\nu^\omega$ is a random measure is a simple exercise. Then we could apply the previous result of Rohlin $\omega$-wise and construct the family $\mu^\omega_a$; properties ii) and iii) above would be obvious but not property i), the joint measurability of $(a,\omega)\mapsto\mu_a^\omega$. To overcome this problem, let us construct a jointly measurable $\mu_a^\omega$ by another procedure and then deduce the other properties. We define $M=\mu^\omega\otimes P$ on the product space $\mathbb R^N\times\Omega$, we consider the map $p:\mathbb R^N\times \Omega\mapsto \mathbb R\times \Omega$ defined as $p(x,\omega)=(\phi(x),\omega)=(\phi\otimes Id)(x,\omega)$, and we can apply the Rohlin disintegration theorem with respect to the partition $\left\{\Gamma_a\times\{\omega\},\ a\in\mathbb R,\ \omega\in\Omega\right\}$. We obtain a unique random family of measures $\left\{\tilde \mu^\omega_a\right\}$ concentrated on the sets of the partition such that
$$\int_{\mathbb R^N\times \Omega} f(x,\omega)dM=\int_{\mathbb R\times\Omega} \left(\int_{\Gamma_a\times\{\omega\}} f(x,\zeta)d\tilde\mu_a^ \omega\right)dp_\sharp M$$
for all positive measurable functions $f:\mathbb{R}^{N}\times\Omega\rightarrow\mathbb{R}$. Now we define the family $\mu^\omega_a$ as
$$\mu_a^\omega(B):=\tilde\mu_a^\omega(B\times\Omega)=\tilde\mu_a^\omega(B\times\{\omega\})$$
for each $\omega\in\Omega$, $a\in\mathbb R$ and $B\in \mathcal B(\mathbb R^N)$. We have to prove that it satisfies properties i) and ii); for $P$-a.e$.$ $\omega\in\Omega$ property iii) above is trivial. We have that the function $(a,\omega)\mapsto\tilde\mu^\omega_a(E)$ is measurable for each set $E\in \mathcal B(\mathbb R^N)\otimes\mathcal F$, and in particular for the sets like $B\times \Omega$, for each $B\in \mathcal B(\mathbb R^N)$: so it is jointly measurable also $(a,\omega)\mapsto \mu^\omega_a(B)$. Moreover,
$$p_\sharp M=(\phi\otimes Id)_\sharp (\mu^\omega\otimes P)=\nu^\omega\otimes P$$
so
$$\int_{\mathbb R^N\times \Omega} f(x,\omega)d\mu^\omega\otimes P=\int_{\mathbb R\times\Omega} \left(\int_{\Gamma_a} f(x,\omega)d\mu_a^\omega\right)d\nu^\omega\otimes P$$
that is the same of
$$E\left[\int_{\mathbb R^N\times \Omega} f(x,\omega)d\mu^\omega-\int_{\mathbb R\times\Omega} \left(\int_{\Gamma_a} f(x,\omega)d\mu_a^\omega\right)d\nu^\omega\right]=0.$$
We can define the event $\Omega_1=\left\{\omega\in\Omega:\ \int_{\mathbb R^N} f(x,\omega)d\mu^\omega>\int_{\mathbb R} \left(\int_{\Gamma_a} f(x,\omega)d\mu_a^\omega\right)d\nu^\omega\right\}$ and the random function $f_1(\omega,\cdot):=f(\omega,\cdot)$ if $\omega\in\Omega_1$ and zero otherwise. Then from the previous equation we have
$$E\left[\left|\int_{\mathbb R^N\times \Omega} f_1(x,\omega)d\mu^\omega-\int_{\mathbb R\times\Omega} \left(\int_{\Gamma_a} f_1(x,\omega)d\mu_a^\omega\right)d\nu^\omega\right|\right]=0$$
hence $P-$a.e$.$
$$\int_{\mathbb R^N\times \Omega} f_1(x,\omega)d\mu^\omega-\int_{\mathbb R\times\Omega} \left(\int_{\Gamma_a} f_1(x,\omega)d\mu_a^\omega\right)d\nu^\omega=0$$
so $\Omega_1$ is negligible. In the same way we can prove that $\Omega_2:=\Big\{\omega\in\Omega:\ \int_{\mathbb R^N} f(x,\omega)d\mu^\omega<\int_{\mathbb R} \left(\int_{\Gamma_a} f(x,\omega)d\mu_a^\omega\right)d\nu^\omega\Big\}$ is negligible too. We obtained that for $P-$a.e$.$ $\omega\in\Omega$ it must be
$$\int_{\mathbb R^N} f(x,\omega)d\mu^\omega=\int_{\mathbb R} \left(\int_{\Gamma_a} f(x,\omega)d\mu_a^\omega\right)d\nu^\omega.$$
In particular we have proved property ii) above, extended to all random test functions.
\end{proof}

\subsection{Non-degeneracy conditions}\label{2_2}
\begin{definition}
\label{def controlled}Given the continuous semimartingale $X$ in
$\mathbb{R}^{N}$, consider the random positive measure $\eta_{X}\left(
dt\right)  $ defined by%
\[
\int_{0}^{T}\varphi\left(  t\right)  \eta_{X}\left(  dt\right)  =\sum
_{i=1}^{N}\int_{0}^{T}\varphi\left(  t\right)  d\left\langle X^{i}%
\right\rangle _{t}+\sum_{i,j=1}^{N}\int_{0}^{T}\varphi\left(  t\right)
d\left\langle X^{i}+X^{j}\right\rangle _{t}%
\]
for all $\varphi\in C\left(  \left[  0,T\right]  \right)  $. Let $A$ be a
Borel set of $\mathbb{R}^{N}$. Let $\phi:\mathbb{R}^{N}\rightarrow\mathbb{R}$
be a Borel function. We say that $\phi\left(  X\right)  $ controls $X$ in
quadratic variation on the set $A$ if
\begin{description}
\item[i)] $\phi\left(  X\right)  $ is a continuous semimartingale

\item[ii)] there is a random constant $C_{A}>0$ such that
\[
\int_{0}^{T}1_{X_{t}\in A}\left\vert \varphi\left(  t\right)  \right\vert
\eta_{X}\left(  dt\right)  \leq C_{A}\int_{0}^{T}1_{X_{t}\in A}\left\vert
\varphi\left(  t\right)  \right\vert d\left\langle \phi\left(  X\right)
\right\rangle _{t}%
\]
for all $\varphi\in C\left(  \left[  0,T\right]  \right)  $.
\end{description}
\end{definition}

The condition that $\phi(X)$ controls $X$ in quadratic variation will be the
main assumption of the multidimensional occupation time formula of the next
section. Now we want to give a very general sufficient condition for it, that we
use several times in the paper.

By $Lip_{loc}\left(  A\right)  $ of an open set $A$ we denote the set
of locally Lipschitz continuous functions on $A$ and we recall that such
functions are differentiable almost everywhere. Moreover, we shall say that
$\left\langle X\right\rangle _{t}$ is Lipschitz continuous if, for each
$j,k=1,...,N$, there is an adapted process with bounded paths $g_{t}^{jk}$
such that, with probability one,
\[
\left\langle X^{j},X^{k}\right\rangle _{t}=\int_{0}^{t}g_{s}^{jk}ds.
\]
If $X$ solves a differential equation of the form $sX_t=b(t,X_t)dt+\sigma(t,X_t)dW_t$ then $g_t=\sigma\sigma^{T}$ and certain assumptions we
impose on $g_{t}$ correspond to usual non-degeneracy assumptions on $\sigma$.

\begin{definition}
\label{definition non-degenerate}Let $X$ be a continuous semimartingale in
$\mathbb{R}^{N}$ and $\phi:\mathbb{R}^{N}\rightarrow\mathbb{R}$ be a Borel
function such that $\phi\left(  X\right)  $ is a continuous semimartingale. Assume that $\left\langle X\right\rangle _{t}$ is Lipschitz continuous and set
$$g_{t}^{jk}=d\left\langle X^{j},X^{k}\right\rangle _{t}/dt$$
such that
$$d\langle \phi(X)\rangle_t=\sum_{i,j=1}^{N}\partial
_{j}\phi\left(  X_{t}\right)  \partial_{k}\phi\left(  X_{t}\right)  g_{t}%
^{jk}dt.$$
Let $A$ be an open set in $\mathbb{R}^{N}$ such that $\phi\in Lip_{loc}\left(  A\right)  $
and let $D_{\phi}\subset A$ be a set of full measure where $\phi$ is
differentiable; assume that, for $P$-a.e. $\omega\in\Omega$, for a.e.
$t\in\left[  0,T\right]  $ the property $X_{t}\left(  \omega\right)  \in A$
implies $X_{t}\left(  \omega\right)  \in D_{\phi}$. We say that $\left\langle
\phi\left(  X\right)  \right\rangle $ is non-degenerate on $A$ if
\[
\inf_{t\in\left[  0,T\right]  :X_{t}\in D_{\phi}}\sum_{i,j=1}^{N}\partial
_{j}\phi\left(  X_{t}\right)  \partial_{k}\phi\left(  X_{t}\right)  g_{t}%
^{jk}>0
\]
with probability one. Notice that this sum is always a non-negative quantity.
\end{definition}

\begin{proposition}
\label{Prop suff cond for control}Let $X$ be a continuous semimartingale in
$\mathbb{R}^{N}$ and $\phi:\mathbb{R}^{N}\rightarrow\mathbb{R}$ be a Borel
function, such that $\langle \phi(X)\rangle$ is non-degenerate on an open set $A$ as described in Definition
\ref{definition non-degenerate}. Then $\phi\left(  X\right)  $
controls $X$ in quadratic variation on $A$.
\end{proposition}

\begin{proof}
We have only to check condition (ii)\ of Definition \ref{def controlled}. We
treat separately each component of $\eta_{X}$ and restrict our proof to a
component of the form $d\left\langle X^{i}+X^{i^{\prime}}\right\rangle _{t}$,
the others being similar. We have%
\begin{align*}
&  \int_{0}^{T}1_{X_{t}\in A}\left\vert \varphi\left(  t\right)  \right\vert
d\left\langle X^{i}+X^{i^{\prime}}\right\rangle _{t}\\
&  =\int_{0}^{T}1_{X_{t}\in A}\left\vert \varphi\left(  t\right)  \right\vert
\frac{\sum_{j,k=1}^{N}\partial_{j}\phi\left(  X_{t}\right)  \partial_{k}%
\phi\left(  X_{t}\right)  g_{t}^{jk}}{\sum_{j,k=1}^{N}\partial_{j}\phi\left(
X_{t}\right)  \partial_{k}\phi\left(  X_{t}\right)  g_{t}^{jk}}\left(
g_{t}^{ii}+g_{t}^{i^{\prime}i^{\prime}}+2g_{t}^{ii^{\prime}}\right)  dt\\
&  \leq C\int_{0}^{T}1_{X_{t}\in A}\left\vert \varphi\left(  t\right)
\right\vert \left\vert \sum_{j,k=1}^{N}\partial_{j}\phi\left(  X_{t}\right)
\partial_{k}\phi\left(  X_{t}\right)  g_{t}^{jk}\right\vert dt
\end{align*}
for a suitable random constant $C>0$ (here we use that $g_{t}^{ii}$ are
bounded and the non-degeneracy assumption for $\left\langle \phi\left(
X\right)  \right\rangle $ on $A$)%
\[
=C\int_{0}^{T}1_{X_{t}\in A}\left\vert \varphi\left(  t\right)  \right\vert
\sum_{j,k=1}^{N}\partial_{j}\phi\left(  X_{t}\right)  \partial_{k}\phi\left(
X_{t}\right)  g_{t}^{jk}dt
\]
(we have used the non-negativity of $\sum_{j,k=1}^{N}\partial_{j}\phi\left(
X_{t}\right)  \partial_{k}\phi\left(  X_{t}\right)  g_{t}^{jk}$)%
\[
=C\int_{0}^{T}1_{X_{t}\in A}\left\vert \varphi\left(  t\right)  \right\vert
d\left\langle \phi\left(  X\right)  \right\rangle _{t}%
\]
and the proof is complete.
\end{proof}

\begin{corollary}\label{cor ipo}
If $cI_{N}\leq g_{t}\leq CI_{N}$ a.s. in $\left(  t,\omega\right)  $, for some constants $C\geq c>0$, $\phi\in
C^{2}\left(  \mathbb{R}^{N}\right)$, and $\inf_{x\in A}\left\vert \nabla
\phi\left(  x\right)  \right\vert >0$ on an open set $A\subset\mathbb{R}^{N}$,
then $\phi\left(  X\right)  $ controls $X$ in quadratic variation on $A$.
\end{corollary}

\begin{proof}
We have that
\begin{align*}
\inf_{t\in\left[  0,T\right]  :X_{t}\in A}\sum_{i,j=1}^{N}\partial_{j}%
\phi\left(  X_{t}\right)  \partial_{k}\phi\left(  X_{t}\right)  g_{t}^{jk}
&=\inf_{t\in\left[  0,T\right]  :X_{t}\in A}\left\vert \nabla\phi\left(
X_{t}\right)  \right\vert ^{2}\\
=\inf_{t\in\left[  0,T\right]  :X_{t}\in A}\left\vert \nabla\phi_{A}\left(
X_{t}\right)  \right\vert ^{2}&\geq\min_{x\in\overline{A}}\left\vert \nabla\phi_{A}\left(  x\right)
\right\vert >0.
\end{align*}
Hence the hypotheses of non-degeneracy of Definition \ref{definition non-degenerate} are satisfied taking $D_\phi=A$ and using the It\^o formula.
\end{proof}

\subsection{Local occupation time formula}\label{2_3}
\begin{theorem}
\label{Thm occupation time formula localized}Let $X$ be a continuous
semimartingale in $\mathbb{R}^{N}$, $\phi:\mathbb{R}^{N}\rightarrow\mathbb{R}$
be a Borel function, $A$ be an open set of $\mathbb{R}^{N}$. Assume that
$\phi\left(  X\right)  $ controls $X$ in quadratic variation on $A$. Then
there exists a random bounded compact support non-negative function
$\mathcal{L}_{i,A,\phi}^{a}$ and random probability measures $Q_{A,\phi}^{i}\left(
a,dx\right)  $, such that $Q_{A,\phi}^{i}\left(  a,\cdot\right)$ is concentrated on
$\Gamma_{a}=\left\{  x\in A:\phi\left(  x\right)  =a\right\}  $ for a.e.
$a\in\mathbb{R}$, and that%
\begin{align*}
\int_{0}^{T}1_{X_{t}\in A}f\left(  X_{t}\right)  d\left\langle
X^{i}\right\rangle _{t}=\int_{\mathbb{R}}\left(  \int_{\Gamma_{a}}f\left(  x\right)  Q_{A,\phi}%
^{i}\left(  a,dx\right)  \right)  \mathcal{L}_{i,A,\phi}^{a}da.
\end{align*}
Moreover we have
\[
\mathcal{L}_{i,A,\phi}^{a}=\lim_{\varepsilon\rightarrow0}\frac{1}{2\varepsilon}%
\int_{0}^{T}1_{X_{t}\in A}1_{(a-\varepsilon,a+\varepsilon)}(\phi
(X_{t}))d\left\langle X^{i}\right\rangle _{t}%
\]
for a.e. $a$. Similar results hold for $\int_{0}^{T}f\left(  X_{t}\right)
d\left\langle X^{i}+X^{j}\right\rangle _{t}$.
\end{theorem}

\begin{proof}
Consider the random Borel measure $\mu_{A}^{i}$ on $\mathbb{R}^{N}$ defined as%
\[
\mu_{A}^{i}\left(  f\right)  =\int_{0}^{T}1_{X_{t}\in A}f\left(  X_{t}\right)
d\left\langle X^{i}\right\rangle _{t}%
\]
for all $f\in C_{b}\left(  \mathbb{R}^{N}\right)  $. We have
\[
\mu_{A}^{i}\left(  A^{c}\right)  =\int_{0}^{T}1_{X_{t}\in A}1_{X_{t}\notin
A}d\left\langle X^{i}\right\rangle _{t}=0
\]
namely $\mu_{A}^{i}$ is concentrated on $A$. From Proposition \ref{rohlin},
there exists a family of probability measures $Q_{A,\phi}^{i}\left(  a,\cdot
\right)  $ concentrated on $\Gamma_{a}$, for $\nu$-a.e. $a\in\mathbb{R}$, and
a Borel measure $\nu_{A,\phi}^{i}$ on $\mathbb{R}$, such that%
\[
\mu_{A}^{i}\left(  f\right)  =\int_{\mathbb{R}}\left(  \int_{\Gamma_{a}%
}f\left(  x\right)  Q_{A,\phi}^{i}\left(  a,dx\right)  \right)  \nu_{A,\phi}^{i}\left(
da\right)
\]
for all positive Borel functions $f:\mathbb{R}^{N}\rightarrow\mathbb{R}$. We
want to prove that $\nu_{A,\phi}^{i}$ has a bounded density with respect to
Lebesgue measure.

If we choose $f$ of the form $f\left(  x\right)  =\theta\left(  \phi\left(
x\right)  \right)  $ with a positive Borel function $\theta:\mathbb{R}%
\rightarrow\mathbb{R}$, we get
\[
\int_{0}^{T}1_{X_{t}\in A}\theta\left(  \phi\left(  X_{t}\right)  \right)
d\left\langle X^{i}\right\rangle _{t}=\int_{\mathbb{R}}\theta\left(  a\right)
\nu_{A,\phi}^{i}\left(  da\right)  .
\]
Thus, let us consider the random linear functional
\[
F_{i,A}\left(  \theta\right)  =\int_{0}^{T}1_{X_{t}\in A}\theta\left(
\phi\left(  X_{t}\right)  \right)  d\left\langle X^{i}\right\rangle _{t}.
\]
We have, by the main assumption,%
\begin{align*}
\left\vert F_{i,A}\left(  \theta\right)  \right\vert  &  \leq\int_{0}%
^{T}1_{X_{t}\in A}\left\vert \theta\left(  \phi\left(  X_{t}\right)  \right)
\right\vert d\left\langle X^{i}\right\rangle _{t}\\
&  \leq C_{A}\int_{0}^{T}1_{X_{t}\in A}\left\vert \theta\left(  \phi\left(
X_{t}\right)  \right)  \right\vert d\left\langle \phi\left(  X\right)
\right\rangle _{t}\leq C_{A}\int_{0}^{T}\left\vert \theta\left(  \phi\left(  X_{t}\right)
\right)  \right\vert d\left\langle \phi\left(  X\right)  \right\rangle _{t}.
\end{align*}
By the occupation time formula for $\phi\left(  X\right)  $,
\[
=C_{A}\int_{\mathbb{R}}\left\vert \theta\left(  a\right)  \right\vert
L_{T}^{a}\left(  \phi\left(  X\right)  \right)  da
\]
where $L_{T}^{a}\left(  \phi\left(  X\right)  \right)  $ is the local time at
$a$ of the continuous semimartingale $\phi\left(  X\right)  $ on $\left[
0,T\right]  $. This local time, as a function of $a$, is, with probability
one, c\`{a}dl\`{a}g and bounded support. Hence%
\[
\leq C_{A}^{\prime}\int_{\mathbb{R}}\left\vert \theta\left(  a\right)
\right\vert da
\]
for a new random constant $C_{A}^{\prime}>0$. The functional $F_{i,A}$ is thus
($\omega$-wise) bounded continuous on $L^{1}\left(  \mathbb{R}\right)  $, and
it is non-negative on non-negative functions, and thus there exists a bounded
non-negative function $\mathcal{L}_{i,A,\phi}^{a}$ such that
\[
F_{i,A}\left(  \theta\right)  =\int_{\mathbb{R}}\theta\left(  a\right)
\mathcal{L}_{i,A,\phi}^{a}da.
\]
This proves the first claim of the theorem.

If we use $\theta=1_{(a-\varepsilon,a+\varepsilon)}$ as a test function, with
a given $a$, we obtain that%
\[
\frac{1}{2\varepsilon}\int_{0}^{T}1_{X_{t}\in A}1_{(a-\varepsilon
,a+\varepsilon)}\left(  \phi\left(  X_{t}\right)  \right)  d\left\langle
X^{i}\right\rangle _{t}=\frac{1}{2\varepsilon}\int_{a-\varepsilon
}^{a+\varepsilon}\mathcal{L}_{i,A,\phi}^{a^{\prime}}da^{\prime}.
\]
Thanks to the Lebesgue theorem, we get
\[
\mathcal{L}_{i,A,\phi}^{a}=\lim_{\varepsilon\rightarrow0}\frac{1}{2\varepsilon}%
\int_{0}^{T}1_{X_{t}\in A}1_{(a-\varepsilon,a+\varepsilon)}\left(  \phi\left(
X_{t}\right)  \right)  d\left\langle X^{i}\right\rangle _{t}
\]
for a.e$.$ $a$. The proof is complete.
\end{proof}

\paragraph{Proof of Theorem \ref{general occup}.}
\begin{proof2}
It readily follows from Corollary \ref{cor ipo} and Theorem \ref{Thm occupation time formula localized}.
\end{proof2}
\refstepcounter{theorem1}

\section{Examples and applications\label{Sect applications}}
\subsection{Singular sets of the functions $\phi$}
The difficulty to apply Theorem \ref{general occup} on the full space $\mathbb{R}^{N}$ is in
the fact that the non-degeneracy assumption (replaced by $\left\vert
\nabla\phi\left(  x\right)  \right\vert >0$) is
quite restrictive. However, in some cases the singular set is
\textit{polar} for the process and the theory applies using Theorem \ref{Thm occupation time formula localized}. Let us see this
\textit{global }corollary.

The set $\left\{  x\in\mathbb{R}^{N}:\left\vert \nabla\phi\left(  x\right)
\right\vert =0\right\}  $ will be called the \textit{singular set} of $\phi$
and denoted by $S_{\phi}$. Recall that $g_t^{jk}:=d\langle X^j,X^k\rangle_t/dt$.

\begin{corollary}
\label{Corollary global}Let $X$ is a continuous semimartingale in $\mathbb R^N$ such that $cI_n\le g_t\le CI_N$ a.s$.$ in $(t,\omega)$ for some constants $C>c>0$. Let $\phi\in C^2(\mathbb R^N)$ be a function such that the singular set $S_{\phi}$ is polar for $X$. Then the results of Theorem \ref{Thm occupation time formula localized} hold.
\end{corollary}
\begin{proof}
The quadratic variation $\left\langle X\right\rangle _{t}$ is obviously
Lipschitz continuous. We have only to check that $\left\langle \phi\left(
X\right)  \right\rangle $ is non-degenerate on the full $\mathbb{R}^{N}$. We
have%
\[
\sum_{i,j=1}^{N}\partial_{i}\phi\left(  X_{t}\right)  \partial_{j}\phi\left(
X_{t}\right)  g_{t}^{ij}\geq c\left\vert \nabla\phi\left(  X_{t}\right)
\right\vert ^{2}.
\]
Given a.s. $\omega$, the function $t\mapsto\left\vert \nabla\phi\left(
X_{t}\left(  \omega\right)  \right)  \right\vert ^{2}$ is continuous
(composition of continuous functions) and different from zero at each point,
since $X_{t}\left(  \omega\right)  $ does not touch the polar set $S_{\phi}$.
Thus, on $\left[  0,T\right]  $, the function $t\mapsto\left\vert \nabla
\phi\left(  X_{t}\left(  \omega\right)  \right)  \right\vert ^{2}$ is strictly positive.

The assumptions of Theorem \ref{Thm occupation time formula localized} hold true and
thus the result holds.
\end{proof}

\begin{example}
If $X$ is a Brownian motion in $\mathbb{R}^{N}$ with $X_{0}=x_{0}$ and $S_{\phi}$ is given by a
finite number of points, with $\nabla\phi\left(  x_{0}\right)  \neq0$, then
the assumptions of Corollary \ref{Corollary global} are satisfied. An example
is
\[
\phi\left(  x\right)  =\left\vert x\right\vert ^{2}%
\]
when $x_{0}\neq0$. If $x_{0}=0$, we have to localize as in Theorem
\ref{general occup}, just by taking $A=\mathbb{R}%
^{N}\backslash\left\{  x:\left\vert x\right\vert \leq\varepsilon\right\}  $,
for some $\varepsilon>0$.
\end{example}

\begin{example}
\label{Example graph 1}If $X$ is a Brownian semimartingale in $\mathbb{R}^{N}$ and the singular
set $S_{\phi}$ is empty, then the assumptions of Corollary
\ref{Corollary global} are satisfied. For example this happens for%
\[
\phi\left(  x\right)  =x^{N}-g\left(  x^{1},...,x^{N-1}\right)
\]
where $g:\mathbb{R}^{N-1}\rightarrow\mathbb{R}$ is a $C^{2}$ function.
Indeed,
\[
\left\vert \nabla\phi\left(  x\right)  \right\vert ^{2}=\left\vert \nabla
g\left(  x^{1},...,x^{N-1}\right)  \right\vert ^{2}+1>0
\]
everywhere. Thus, in the case of a Brownian semimartingale, we may apply the
theory. This is related to \cite{Peskir} (which is much more general). The
manifolds $\Gamma_{a}$ are translations of the graph of $g$.
\end{example}

\subsection{Integration of functions $f$ with singularities}
The question treated in this section is when, for $i=1,...,N$,%
\begin{equation}
P\left(  \int_{0}^{T}\left\vert f\left(  X_{t}\right)  \right\vert
d\left\langle X^{i}\right\rangle _{t}<\infty\right)  =1
\label{finite integral}%
\end{equation}
for functions $f$ which are not bounded. Let us distinguish two cases:
\begin{description}
\item[i)] we already have a foliation $\Gamma_{a}:=\{\phi(x)=a\}$, $a\in\mathbb{R}$,
and $f$ is not bounded only in the transversal direction to the foliation

\item[ii)] we have only one manifold $\Gamma$ and a function $f$ which is unbounded
only in the neighborhood of $\Gamma$.
\end{description}
In the first case we only need to apply the formula;\ in the second case we
have to construct a suitable foliation.

Concerning case (i), we give two examples:\ Example \ref{Ex 1} is
elementary and global, Example \ref{Ex 2} is its general version.

\begin{example}
\label{Ex 1} Let $X$ be a continuous semimartingale in $\mathbb R^N$. Assume $X$ has Lipschitz continuous quadratic variation with
$$P\left(  \inf_{t\in\left[  0,T\right]  }g_{t}^{NN}>0\right)  =1.$$
Let $f:\mathbb{R}^{N}\rightarrow\mathbb{R}$ be a function of the form
\[
f(x)=f_{1}(x)f_{2}(x_{N})
\]
$x=\left(  x_{1},...,x_{N}\right)  \in\mathbb{R}^{N}$, where $f_{1}\in
C(\mathbb{R}^{N})$ and $f_{2}:\mathbb{R}\rightarrow\mathbb R$ is a locally integrable function. Then
(\ref{finite integral}) holds. To prove this fact we use $\phi:\mathbb{R}%
^{N}\rightarrow\mathbb{R}$ defined as $\phi\left(  x\right)  =x_{N}$: one has
\[
\sum_{i,j=1}^{N}\partial_{j}\phi\left(  X_{t}\right)  \partial_{k}\phi\left(X_{t}\right)  g_{t}^{jk}=g_{t}^{NN}
\]
and thus the assumptions of Theorem \ref{Thm occupation time formula localized} hold with $A=\mathbb R^N$.
Therefore
\begin{align*}
\int_{0}^{T}\left\vert f\left(  X_{t}\right)  \right\vert d\left\langle
X^{i}\right\rangle _{t}  &  \leq C_{1}\int_{0}^{T}\left\vert f_{2}\left(
X_{t}^{N}\right)  \right\vert d\left\langle X^{i}\right\rangle _{t}\\
&  =C_{1}\int_{\mathbb{R}}\left(  \int_{\Gamma_{a}}\left\vert f_{2}\left(
x_{N}\right)  \right\vert Q^{i}_{A,\phi}\left(  a,dx\right)  \right)  \mathcal{L}%
_{i,A,\phi}^{a}da
\end{align*}
$$=C_{1}\int_{\mathbb{R}}\left\vert f_{2}(a)\right\vert \mathcal{L}_{i,A,\phi}^{a}da=C_{1}\int_{B_{i}}\left\vert f_{2}(a)\right\vert \mathcal{L}_{i,A,\phi}^{a}da\leq C_{1}C_{2}\int_{B_{i}}\left\vert f_{2}(a)\right\vert da<+\infty$$
where, denoted by $K$ and $B_{i}$ random compact sets containing respectively
the image of the curve $X$ and the support of $\mathcal{L}_{i,A,\phi}^{a}$, we have
set $C_{1}=\sup_{K}\left\vert f_{1}\right\vert $, $C_{2}=\sup_{B_{i}%
}\mathcal{L}_{i,A,\phi}^{\cdot}$.
\end{example}

The following simple application of the occupation time formula will be used below. For instance, in the case when $\langle X^i\rangle_t=t$ it implies that $X_t\notin A\cap \phi^{-1}(N)$ for a.e$.$ $t\in[0,T]$ with probability 1.

\begin{lemma}
\label{corollary measure zero on Gamma}Under the assumptions and with the
notations of Theorem \ref{Thm occupation time formula localized}, with probability one,
\[
\int_{0}^{T}1_{A\cap\phi^{-1}\left(  E\right)  }\left(  X_{t}\right)
d\left\langle X^{i}\right\rangle _{t}=0
\]
for every Borel set $E\subset\mathbb{R}$ of zero Lebesgue measure.
\end{lemma}

\begin{proof}
By the local occupation time formula of Theorem \ref{Thm occupation time formula localized} we have
\begin{align*}
\int_{0}^{T}1_{A\cap\phi^{-1}\left(  E\right)  }\left(  X_{t}\right)
d\left\langle X^{i}\right\rangle _{t}  & =\int_{\mathbb{R}}\left(  \int%
_{\phi^{-1}\left(  a\right)  }1_{A\cap\phi^{-1}\left(  E\right)  }\left(
x\right)  Q_{A,\phi}^{i}\left(  a,dx\right)  \right)  \mathcal{L}_{i,A,\phi}^{a}da\\
& \leq\int_{\mathbb{R}}1_{E}\left(  a\right)  \mathcal{L}_{i,A,\phi}^{a}da=0.
\end{align*}

\end{proof}

If $X$ is a Brownian semimartingale in $\mathbb{R}^{N}$ and we apply Lemma \ref{corollary measure zero on Gamma} we obtain
\[
\int_{0}^{T}1_{A\cap\phi^{-1}\left(  E\right)  }\left(  X_{t}\right)  dt=0
\]
for every Borel set $E\subset\mathbb{R}$ of zero Lebesgue measure. Due to
this, the assumption in Example \ref{Ex 1} that $f_{2}:\mathbb{R\rightarrow R}$ is a
locally integrable function may be replaced by the assumption $f_{2}\in
L_{loc}^{1}\left(  \mathbb{R}\right)  $;\ in other words, the result does not
change if we modify $f_{2}$ on a zero-measure set or in the case when $f_{2}$
is not even defined on a zero-measure set. Indeed if $f_{2}$ is not defined on the zero-measure set $E$, then $f\left(
x\right)  =1_{A}f_{1}\left(  x\right)  f_{2}\left(  \phi\left(  x\right)
\right)  +1_{A^{c}}g\left(  x\right)  $ is not defined on the set $A\cap
\phi^{-1}\left(  E\right)  $. But, with probability one, $X_{t}\in\left(
A\cap\phi^{-1}\left(  E\right)  \right)  ^{c}$ for a.e. $t\in\left[
0,T\right]  $ and thus $\int_{0}^{T}\left\vert f\left(  X_{t}\right)
\right\vert dt$ is well defined. This is the integral $\int_{0}^{T}\left\vert
f\left(  X_{t}\right)  \right\vert d\left\langle X^{i}\right\rangle _{t}$
examined by Example \ref{Ex 1}.

\begin{example}
\label{Ex 2}Let $f:\mathbb{R}^{N}\rightarrow\mathbb{R}$ be a function of the form%
\[
f(x)=1_{A}f_{1}(x)f_{2}\left(  \phi(x)\right)  +1_{A^{c}}g\left(  x\right)
\]
where $A$ is an open set, $f_{1}\in C(\overline{A})$, $f_{2}\in L_{loc}%
^{1}(\mathbb{R})$, $g\in C(A^{c})$, $\phi:\mathbb R^N\rightarrow \mathbb R$ is such that $\phi(X)$ controls $X$ in quadratic variation on $A$. Let $X$ have a Lipschitz continuous quadratic variation and assume that $\langle \phi(X)\rangle$ is non-degenerate on $A$. Then (\ref{finite integral})
holds. Indeed, from Theorem \ref{Thm occupation time formula localized},
\begin{align*}
&  \int_{0}^{T}\left\vert f\left(  X_{t}\right)  \right\vert d\left\langle
X^{i}\right\rangle _{t}\\
&  =\int_{0}^{T}1_{X_{t}\notin A}\left\vert f\left(  X_{t}\right)  \right\vert
d\left\langle X^{i}\right\rangle _{t}+\int_{\mathbb{R}}\left(  \int%
_{\Gamma_{a}}\left\vert f\left(  x\right)  \right\vert Q_{A,\phi}^{i}\left(
a,dx\right)  \right)  \mathcal{L}_{i,A,\phi}^{a}da\\
&  \leq C_{1}\left\langle X^{i}\right\rangle _{T}+C_{2}\int_{B_{i,A}%
}\left\vert f_{2}(a)\right\vert \mathcal{L}_{i,A,\phi}^{a}da
\end{align*}
where, denoted by $K$ and $B_{i,A}$ random compact sets containing
respectively the image of the curve $X$ and the support of $\mathcal{L}%
_{i,A,\phi}^{a}$, we have set $C_{1}=\sup_{K\cap A^{c}}\left\vert g\right\vert $,
$C_{2}=\sup_{K\cap\overline{A}}\left\vert f_{1}\right\vert $. We conclude as
in the previous example.
\end{example}

Let us see now examples of case (ii) above, namely when we have a function $f$
which is singular only along an $\left(  N-1\right)  $-dimensional manifold
$\Gamma$. The problem here is to construct a suitable function $\phi$. Let us
see first a case which relates to \cite{Peskir}.

\begin{example}
Let us continue Example \ref{Example graph 1}. We assume that $X$ is a
Brownian semimartingale and $f:\mathbb{R}^{N}\backslash\Gamma\rightarrow
\mathbb{R}$ is a continuous function, where $\Gamma$ is the graph of a $C^{2}$
function $g:\mathbb{R}^{N-1}\rightarrow\mathbb{R}$. Consider the function%
\[
\phi\left(  x\right)  =x^{N}-g\left(  x^{1},...,x^{N-1}\right)
\]
the associated sets $\Gamma_{a}$ and the numbers $M_{a}^{\phi}\left(  \left\vert f\right\vert \right)  =\max_{x\in\Gamma_{a}}\left\vert f\left(  x\right)  \right\vert$ for every $a\neq0$. If
\[
\int_{-1}^{1}M_{a}^{\phi}\left(  \left\vert f\right\vert \right)  da<\infty
\]
then (\ref{finite integral}) holds. To prove this claim it is sufficient to apply the result of Example \ref{Ex 2}, with
$$A=\left\{x:\ -1<\phi(x)<1\right\},\quad g=f|A^c,\quad f_2(a)= M^\phi_a(|f|)$$
$$f_1(x)=\left\{
  \begin{array}{ll}
    f(x)/f_2(\phi(x)) & \hbox{if}\quad f_2(\phi(x))>0 \\
    0 & \hbox{if}\quad f_2(\phi(x))=0
  \end{array}
\right.$$
\end{example}

In the proof of Theorem \ref{integ} we will use some geometric results that we will develop in Section \ref{Sect distance with sign}. It is presented here because of its conceptual unity with the previous. Let $\Gamma$ be a manifold in $\mathbb R^N$ and $f:\mathbb R^N\backslash \Gamma\rightarrow \mathbb R$ be a measurable function. For every $a, R>0$ define
$$M_{a,R}(|f|):=\sup_{x\in\Gamma^d_{a,R}}|f(x)|$$
where $\Gamma^d_{a,R}=\left\{x\in B(0,R):\ d(x,\Gamma)=a\right\}$ and $B(0,R)$ is a ball.
\begin{definition}
We say that $f:\mathbb R^N\backslash \Gamma\rightarrow \mathbb R$ is in $L^1_{loc}(\Gamma^\bot)$ if for every $R>0$
$$\int_0^1 M_{a,R}(|f|)da<\infty$$
and $f$ is bounded on every compact set in $\mathbb R^N\backslash\Gamma$.
\end{definition}

\paragraph{Proof of Theorem \ref{integ}.}
\begin{proof2}
Let $\mathcal{U}$ and $\delta_{\Gamma}$ be given by Proposition \ref{def dist sign}. Let
$\mathcal{V}$, with $\overline{\mathcal{V}}\subset\mathcal{U}$, and $\phi\in
C^{2}\left(  \mathbb{R}^{N}\right)  $ (extension of $\delta_{\Gamma}$)\ be
given by Corollary \ref{trivial}. We have that that $\phi\left(  X\right)
$ controls $X$ in quadratic variation on $\mathcal{V}$ and in particular on $A:=\mathcal V\cap\{x:\ -1<d(x,\Gamma)<1\}$. From Theorem \ref{general occup} we have
\begin{align*}
&  \int_{0}^{T}\left\vert f\left(  X_{t}\right)  \right\vert d\left\langle
X^{i}\right\rangle _{t}\\
&  =\int_{0}^{T}1_{X_{t}\notin A}\left\vert f\left(  X_{t}\right)  \right\vert
d\left\langle X^{i}\right\rangle _{t}+\int_{\mathbb{R}}\left(  \int%
_{\Gamma_{a}}\left\vert f\left(  x\right)  \right\vert Q_{A,\phi}^{i}\left(
a,dx\right)  \right)  \mathcal{L}_{i,A,\phi}^{a}da.
\end{align*}
For each fixed $\omega\in\Omega$ the trajectory of the process $X$ remains inside a compact ball $B(0,R_\omega)$ and
$$\int_{0}^{T}1_{X_{t}\notin A}\left\vert f\left(  X_{t}\right)  \right\vert d\left\langle X^{i}\right\rangle _{t}(\omega)\le\int_{0}^{T}1_{\overline{A^c\cap B(0,R_\omega)}}\left\vert f\left(  X_{t}\right)  \right\vert d\left\langle X^{i}\right\rangle _{t}(\omega)<\infty$$
because $\overline{A^c\cap B(0,R_\omega)}$ is a compact set in $\mathbb R^N\backslash\Gamma$. Moreover
$$\int_{\mathbb{R}}\left(  \int_{\Gamma_{a}}\left\vert f\left(  x\right)  \right\vert Q_{A,\phi}^{i}\left(a,dx\right)  \right)  \mathcal{L}_{i,A,\phi}^{a}da(\omega)=\int_{\mathbb{R_+}}\left(  \int_{\Gamma^d_{a,R}}\left\vert f\left(  x\right)  \right\vert Q_{A,\phi}^{i}\left(a,dx\right)  \right)  \mathcal{L}_{i,A,\phi}^{a}da(\omega)$$
using that for each $a>0$ we have $\Gamma_a\cup\Gamma_{-a}=\Gamma^d_{a,R}$. So it is
$$\le \int_{\mathbb{R_+}}\left(  \int_{\Gamma^d_{a,R}}M_{a,R}(|f|) Q_{A,\phi}^{i}\left(a,dx\right)  \right)  \mathcal{L}_{i,A,\phi}^{a}da(\omega)=\int_{\mathbb{R_+}}\left( M_{a,R}(|f|)\right)  \mathcal{L}_{i,A,\phi}^{a}da(\omega)<\infty$$
because $\mathcal{L}_{i,A,\phi}^{a}$ is bounded.
\end{proof2}
\refstepcounter{theorem1}

\begin{remark}\label{R integ}
In the case of a Brownian motion $B$, using its explicit Gaussian density one can show that $P\left(\int_0^T|f(B_t)|dt<\infty\right)=1$ is true for functions $f$ of class $L^q_{loc}(\mathbb R^N)$ with $q>\frac{N}{2}\vee 1$ (see for instance \cite{Kry-Ro}). In a sense, the previous theorem gives us a more precise result, valid for all Brownian semimartingales and when the singularity set of $f$ is of a special type.
\end{remark}

\subsection{SDEs with singular coefficients}
The idea of the following example is taken from Cerny-Engelbert \cite{Cerny
Engelbert}, where a similar case is treated in dimension one.

The problem from which the example arises is to construct an \textit{example
of non-existence} for an SDE in $\mathbb{R}^{N}$ of the form%
\[
dX_{t}=b\left(  X_{t}\right)  dt+dW_{t},\qquad X_{0}=x_{0}%
\]
outside the present classes of $b$'s where existence is known, in order to
test the sharpness of such classes. We refer to the very general
result of \cite{Kry-Ro} which states that strong (local) existence (and
pathwise uniqueness) is known when $b\in L_{loc}^{p}\left(  \mathbb{R}%
^{N},\mathbb{R}^{N}\right)  $ for some $p>N\vee2$. Also the result of
existence of weak solutions of \cite{BassChen} for distributional drift, when
particularized to distributions realized by functions, gives the same class.
Thus it looks optimal, even for weak existence.

The function%
\[
b\left(  x\right)  =C\left\vert x\right\vert ^{-2}x
\]
is of class $L_{loc}^{p}\left(  \mathbb{R}^{N},\mathbb{R}^{N}\right)  $ only
for $p<N$, thus it is outside the boundary of the previous theory. We prove
that, in the particular case $C=-\frac{1}{2}$ and $x_{0}=0$, no weak solution exists.

First, by weak solution $(X,W)$ on $\left[  0,T\right]  $ (on a local random
time interval the argument is similar) we mean that there is a filtered
probability space $\left(  \Omega,\mathcal{A},\mathcal{F}_{t},P\right)  $, an
$\mathcal{F}_{t}$-Brownian motion $W$ in $\mathbb{R}^{d}$, an $\mathcal{F}%
_{t}$-adapted continuous process $\left(  X_{t}\right)  _{t\geq0}$ in
$\mathbb{R}^{d}$, such that
\[
\int_{0}^{T}\left\vert b\left(  X_{t}\right)  \right\vert dt<\infty
\]
and, a.s.,
\[
X_{t}=\int_{0}^{t}b\left(  X_{s}\right)  ds+W_{t}.
\]
Hence $X$ is a continuous semimartingale, with quadratic covariation
$\left\langle X^{i},X^{j}\right\rangle _{t}=\delta_{ij}t$ between its
components. Take (to be more precise than above)%
\[
b\left(  x\right)  =\left\{
\begin{array}
[c]{ccc}%
0 & \text{if} & x=0\\
-\frac{1}{2\left\vert x\right\vert ^{2}}x & \text{if} & x\neq0
\end{array}
\right.  .
\]
We shall write $b\left(  x\right)  =-\frac{1_{x\neq0}}{2\left\vert
x\right\vert ^{2}}x$.

\begin{proposition}
The SDE\ with this vector field $b$ and $x_{0}=0$ does not have any weak solution $(X,W)$.
\end{proposition}

\begin{proof}
Assume, by contradiction, that $\left(  X,W\right)  $ is a weak solution. By
It\^{o} formula,
\begin{align*}
d\left\vert X_{t}\right\vert ^{2}  &  =-1_{X_{t}\neq0}dt+2X_{t}\cdot
dW_{t}+dt=1_{X_{t}=0}dt+2X_{t}\cdot dW_{t}.
\end{align*}
From Theorem \ref{general occup} with $A=\mathbb{R}^{N}$
and $\phi\left(  x\right)  =x_{1}$, we get%
\[
\int_{0}^{T}1_{X_{t}=0}dt=\int_{\mathbb{R}}\left(  \int_{\Gamma_{a}%
}1_{\left\{  0\right\}  }\left(  x\right)  Q_{A,\phi}\left(  a,dx\right)  \right)
\mathcal{L}_{A,\phi}^{a}da\leq\int_{\mathbb{R}}\eta\left(  a\right)  \mathcal{L}%
_{A,\phi}^{a}da=0
\]
where $\eta\left(  0\right)  =Q_{A,\phi}\left(  0,\left\{  0\right\}  \right)
\leq1$, $\eta\left(  a\right)  =0$ for $a\neq0$ and we omitted the (identical) dependence by $i\in 1,\dots, N$. Hence%
\[
\left\vert X_{t}\right\vert ^{2}=\int_{0}^{t}2X_{s}\cdot dW_{s}.
\]
Therefore $\left\vert X_{t}\right\vert ^{2}$ is a positive local martingale,
null at $t=0$. This implies that $\left\vert X_{t}\right\vert ^{2}\equiv0$
hence $X_{t}\equiv0$. But this contradicts the fact that $\left\langle
X^{i},X^{j}\right\rangle _{t}=\delta_{ij}t$.
\end{proof}

\begin{remark}
The property $\int_{0}^{T}1_{X_{t}=0}dt=0$, where we have used our multidimensional occupation
time formula, can be proved also in other ways. The point of this example is
not to show a striking application where the occupation time formula is
strictly necessary but an example where it can be used to prove something
useful and non-trivial in a line.
\end{remark}

\section{Embedding of a manifold $\Gamma$ in a foliation}\label{Sect embed}
In some examples we have a process $X$ in $\mathbb{R}^{N}$, a
function $f:\mathbb{R}^{N}\rightarrow\mathbb{R}$ for which we want to consider
$\int_{0}^{T}f\left(  X_{t}\right)  d\left\langle X^{i}\right\rangle _{t}$,
and an $(N-1)-$dimensional manifold $\Gamma$ in $\mathbb{R}^{N}$ where $f$ is singular (see Section \ref{Sect applications}).

In this section we pose the problem to embed $\Gamma$ in a foliation $\left\{
\Gamma_{a};\ a\in\mathbb{R}\right\}$ given by level sets of some function
$\phi:\mathbb{R}^{N}\rightarrow\mathbb{R}$, satisfying the hypotheses of Theorem \ref{Thm occupation time formula localized}. In order to solve this problem, we propose to use as $\phi$ a smooth
extension of the \textit{signed distance function}. This requires that
$\Gamma$ is an orientable manifold with some other simple properties. The
signed distance function then exists of class $C^{2}$ in a neighborhood of
$\Gamma$ and we may extend it on the full space, still of class $C^{2}$, equal
to zero outside a larger neighborhood.

This construction is however insufficient to prove Theorem \ref{Thm L=L=L}, which is the second aim of this section. An extension
equal to zero does not work because it affects the sets $\phi^{-1}\left(
a\right)  $ for small $a$ (to be precise, Proposition \ref{good}, part (iii) would
fail and it plays a basic role in the proof of Theorem \ref{Thm L=L=L}, see identity (\ref{identityA})).
We then choose to work with the \textit{distance function} $d\left(
\cdot,\Gamma\right)  $, suitably extended outside a neighborhood $\mathcal{V}$
of $\Gamma$. Since it is not smooth on $\Gamma$, we have to overcome some
difficulties. At the end we develop the necessary geometric preliminaries to
prove Theorem \ref{Thm L=L=L} and later on Theorem \ref{Thm L=L} in Section \ref{2}.

\subsection{Construction of $\phi$ given a manifold $\Gamma$ of class $C^2$}\label{2_4}
Given the manifold $\Gamma$, there are several functions $\phi$ such that $\left\{
\phi=0\right\}=\Gamma$. Basic requisite for us is that $\phi\left(
X\right)  $ is a continuous semimartingale. This is clearly achieved if $\phi$
is of class $C^{2}$ but there are interesting cases in which, in order to have
other properties of $\phi$, we have to give up with a full $C^{2}$-regularity.

The second requisite on $\phi$ is a form on non-degeneracy in order to have
that $\phi\left(  X\right)  $ controls $X$ in quadratic variation on a set of
interest $A$ (typically a neighborhood of $\Gamma$). The key for non-degeneracy is
that $\nabla\phi$ should not vanish (not too much)\ in $A$. Along with the
requirement $\left\{  \phi=0\right\}  =\Gamma$, this means that we have to
look for non-trivial functions $\phi$.

Finally, in Section \ref{Sect distance with sign}, we shall see the advantage
of having further properties, related to the \textit{eikonal equation}%
\[
\left\vert \nabla\phi\left(  x\right)  \right\vert =1
\]
(aimed to hold at least locally around $\Gamma$). Thus we pose in this section
the following question: given a manifold $\Gamma$, construct a function
$\phi:\mathbb{R}^{N}\rightarrow\mathbb{R}$ such that
\begin{itemize}
  \item $\left\{  \phi=0\right\}=\Gamma$,
  \item $\phi\left(  X\right)  $ is a continuous semimartingale,
  \item $\left\vert \nabla\phi\right\vert =1$ in a neighborhood of $\Gamma$.
\end{itemize}

Natural candidates are the distance function, $x\mapsto d\left(  x,\Gamma\right)  $
and the signed distance function, when defined. The advantage of the distance
function is that it is globally and elementary defined in full generality on
$\Gamma$, but when $x$ crosses $\Gamma$ it is not differentiable (it is also
non-differentiable far from $\Gamma$, but this is less relevant). The drawback
of the signed distance function is first of all the difficulty to define it,
but then it has the advantage of some smoothness also around $\Gamma$. But let us first mention a case when the signed distance
has an obvious definition.

\begin{example}
\label{Example boundary}Let $D$ be a non-empty open set in $\mathbb{R}^{N}$
with non-empty complementary set $D^{c}$. Let $\Gamma$ be the boundary of $D$.
We call signed distance function from $\Gamma$ the function%
\[
\delta_{\Gamma}\left(  x\right)  =\left\{
\begin{array}
[c]{ccc}%
d\left(  x,\Gamma\right)   & \text{if} & x\in D\\
-d\left(  x,\Gamma\right)   & \text{if} & x\in D^{c}%
\end{array}
\right.  .
\]
If $\Gamma$ is piecewise smooth, then $\delta_{\Gamma}$ is differentiable a.e.
and, where it is differentiable, it satisfies $\left\vert \nabla\delta
_{\Gamma}\left(  x\right)  \right\vert =1$.\ If $\Gamma$ is of class $C^{2}$,
then there is a neighborhood $\mathcal{U}$ of $\Gamma$ where $\delta_{\Gamma}$ is
of class $C^{2}$, and the neighborhood can be taken of the form%
\[
\mathcal{U}_{\varepsilon}=\left\{  x\in\mathbb{R}^{N}:d\left(  x,\Gamma
\right)  <\varepsilon\right\}
\]
for some $\varepsilon>0$ if $D$ is bounded. A discussion of these and other
properties can be found in \cite{GilbTrud}.
\end{example}

Inspired by these properties, let us axiomatize some properties of a signed
distance function from a general manifold, so that it will be more clear what
we use in each general statement. For the definition of embedded manifold see \cite{Kosinski}.

\begin{notation}
If $\Gamma$ is an $(N-1)$-dimensional orientable manifold of class $C^{2}$, closed and without boundary and embedded in $\mathbb R^N$, then we call it a leaf-manifold.
\end{notation}

\begin{proposition}\label{def dist sign}
Let $\Gamma$ be a leaf-manifold. Then there exist an open neighborhood $\mathcal{U}$ of $\Gamma$ and a function $\delta_{\Gamma}%
:\mathcal{U}\rightarrow\mathbb{R}$ such that:
\begin{description}
  \item[i)] $\delta_{\Gamma}\in C^{2}\left(  \mathcal{U}\right)  $
  \item[ii)] $\left\vert \delta_{\Gamma}\left(  x\right)  \right\vert =d\left(x,\Gamma\right)  $ for all $x\in\mathcal{U}$.
\end{description}
This (not unique) function is called signed distance function $\delta_{\Gamma}$ on the open neighborhood $\mathcal{U}$ of $\Gamma$.
\end{proposition}
Proof in Appendix A.

\begin{lemma}\label{prop dist}
Following Proposition \ref{def dist sign} assume that the manifold $\Gamma$ has a signed distance function $\delta_{\Gamma}$ on
$\mathcal{U}$. Then one has:
\begin{description}
  \item[iii)] on every connected component of $\mathcal{U}\backslash\Gamma$, the
function $\delta_{\Gamma}\left(  \cdot\right)  $ is either identically equal
to $d\left(  \cdot,\Gamma\right)  $ or to $-d\left(  \cdot,\Gamma\right)  $
  \item[iv)] $d\left(  x,\Gamma\right)  $ is of class $C^{2}$ on $\mathcal{U}%
\backslash\Gamma$
  \item[v)] each $x\in\mathcal{U}$ has a unique point
$P_{\Gamma}\left(  x\right)  \in\Gamma$ of minimal distance.
  \item[vi)] $\left\vert \nabla\delta_{\Gamma}\left(  x\right)  \right\vert =1$ for all
$x\in\mathcal{U}$
\end{description}
\end{lemma}
Proof in Appendix A.

\begin{remark}
In Proposition \ref{def dist sign} we imposed that $\Gamma$ has no boundary since it would be incompatible
with the required properties just in the simplest case of $\Gamma$ equal to the closed $(N-1)$-dimensional disk. Indeed in general if $\mathcal U\backslash \Gamma$ has only a single connected component, by property (iii) we would have that $\delta_\Gamma$ is equal to $d\left(  \cdot,\Gamma\right)$ (or to $-d\left(  \cdot,\Gamma\right)$) in the whole $\mathcal U\backslash \Gamma$, so by continuity it is true also in $\mathcal U$, and this is a contradiction with its $C^2$ regularity.
\end{remark}

The first aim of this subsection was to construct a function $\phi$, out of a
manifold $\Gamma$, in order to apply Theorem \ref{general occup}. If $\Gamma$ is a leaf
manifold, we have solved this problem. Indeed, let $\mathcal{U}$ and
$\delta_{\Gamma}$ be given by Proposition \ref{def dist sign}. Let $\mathcal{V}$ be a
neighborhood of $\Gamma$ such that $\overline{\mathcal{V}}\subset\mathcal{U}$.
There exists a $C^{2}$ function $\phi:\mathbb{R}^{N}\rightarrow\mathbb{R}$
such that $\phi=\delta_{\Gamma}$ on $\mathcal{V}$ and $\phi=0$ outside
$\mathcal{U}$. Take $A=\mathcal{V}$. Then, by Lemma \ref{prop dist} part (vi) we have
$\inf_{x\in A}\left\vert \nabla\phi\left(  x\right)  \right\vert >0$. All the
assumptions of Theorem \ref{general occup} hold true. To summarize:

\begin{corollary}\label{trivial}
Let $X$ is a continuous semimartingale in $\mathbb R^N$ such that $cI_n\le g_t\le CI_N$ a.s$.$ in $(t,\omega)$ for some constants $C>c>0$. Let $\Gamma$ be leaf-manifold. Let $\mathcal U$ be given by Proposition \ref{def dist sign} and let $\mathcal V$ be an open neighborhood of $\Gamma$ such that $\overline{\mathcal V}\subset \mathcal U$. Let $\phi\in C^{2}\left(  \mathbb{R}^{N}\right)  $ be an extension of $\delta_{\Gamma}$ from $\overline {\mathcal V}$. Then the assumptions of Theorem \ref{general occup} are satisfied with $A=\mathcal V$.
\end{corollary}

\subsubsection{Uniform neighborhoods $\mathcal U_\varepsilon(\Gamma)$}\label{unif}
We investigate a slightly more restrictive condition since it will turn out to be relevant in Section \ref{2}, and useful in some of the next proofs.

\begin{notation}
Given $\varepsilon>0$ and a set $S$, we denote by $\mathcal{U}_{\varepsilon}\left(
S\right)  $ the open neighborhood of $S$ or radius $\varepsilon$: the set
of all $x\in\mathbb{R}^{N}$ such that $d\left(  x,S\right)  <\varepsilon$.
\end{notation}

\begin{remark}
When $\mathcal U=\mathcal U_\varepsilon(\Gamma)$ in Proposition \ref{def dist sign}, we may choose $\mathcal V=\mathcal U_{\varepsilon_1}(\Gamma)$ with $\varepsilon_1<\varepsilon$, in Corollary \ref{trivial}. We shall always make this choice, below.
\end{remark}

Not all $C^2$ orientable manifolds $\Gamma$ have a ``uniform'' neighborhood of the form
$\mathcal U_\varepsilon (\Gamma)$ satisfying the conditions of Proposition \ref{def dist sign}. For instance, in $\mathbb R^2$, the
graph of the function $y = \sin x^2$ has a neighborhood $\mathcal U$ as in Proposition \ref{def dist sign} but not
a neighborhood of the form $\mathcal U_\varepsilon (\Gamma)$ ($\mathcal U$ has to shrink at infinity). Anyways if $\Gamma$ is compact, we can always take a neighborhood $\mathcal V$ of $\Gamma$ such that it is bounded and $\overline V\subseteq \mathcal U$; hence we can define $\varepsilon_0:=\max_{x\in \overline{\mathcal V}} d(x,\Gamma)$ and restrict the signed distance on the set $\mathcal U_{\varepsilon_0} (\Gamma)\subset \mathcal U$.

We discuss now a general class of manifolds which fulfill Proposition \ref{def dist sign} with $\mathcal U=\mathcal U_\varepsilon(\Gamma)$. A subset $S$ of $\mathbb R^N$ is called proximally smooth, or with positive reach, if exists $\varepsilon>0$ such that for each $x\in\mathcal{U}_{\varepsilon}(S)$ (defined above) there exists a unique minimizer of the distance function from $x$ to $S$ (see also \cite{Crasta}). This number $\varepsilon$ is called the reach of the set $S$.

We remind here the Mises theorem (see \cite{Zajiek}): it states that for each closed set $F\in \mathbb R^N$ the one-sided directional derivatives $D_v d(x,F):=\lim_{\varepsilon\rightarrow 0+} \frac{d(x+\varepsilon v,F)-d(x,F)}{\varepsilon}$ are well defined for all $x\in\mathbb R^N\backslash F$, $v\in\mathbb R^N$. In particular if we call $P_F(x)$ the set of metric projections of $x\notin F$ on $F$, we have that $\forall v\in\mathbb R^N$,
\begin{align}\label{ceco}
D_v d(x,F)=\inf\left\{\frac{v\cdot(x-y)}{|x-y|},\ y\in P_F(x)\right\}.
\end{align}

Using (\ref{ceco}) we obtain that the distance from a set $S$ with positive reach is of class $C^1$ in $\mathcal U_\varepsilon(S)\backslash S$ (indeed the function $P_F$ is continuous). But in the case of $S$ also being a leaf-manifold, we can prove $C^2$ regularity.

\begin{proposition}\label{posit}
Let $\Gamma$ be a leaf-manifold with positive reach $\varepsilon$. Then it is possible to define a signed distance on the open neighborhood $\mathcal U_\varepsilon(\Gamma)$.
\end{proposition}
Proof in Appendix A.

The proximally smooth sets were introduced in 1959 in the seminal paper \cite{Federer} by Federer, who also proved many of their most relevant properties. A proximally smooth set must be closed, and the class contains the convex sets, as well as those sets which can be defined locally by means of finitely many equations $f(x)=0$ and inequalities $f(x)\le 0$ using real valued continuously differentiable functions $f$ whose gradients are Lipschitz continuous and satisfy a certain independence condition (Thm 4.12, \cite{Federer}).

 \begin{theorem}[Federer] Suppose $f_i, \dots, f_m$ are continuously differentiable real valued functions on an open subset of $\mathbb R^N$, $\nabla f_i, \dots, \nabla f_m$ are Lipschitz continuous for each $0\le k\le m$, and
$$A=\bigcap^k_{i=1}\left\{x\ :f_i(x)=0\right\}\ \cap\ \bigcap^m_{i=k+1} \left\{x\ :f_i(x)\le 0\right\}.$$
If $\forall a\in A$, we take $J=\{i:\ f_i(a)=0\}$, and there do not exist real numbers $t_i$, corresponding to $i\in J$, such that $t_i\neq 0$ for some $i\in J$, $t_i\ge 0$ whenever $i\in J$, $i>k$, and
$$\sum_{i\in J} t_i\nabla f_i(a)=0$$
then $A$ has positive reach.
\end{theorem}

\subsubsection{Good extension of $d\left(\cdot,\Gamma\right)$}\label{goodex}
As we said at the beginning of the section, the extension of $\delta_{\Gamma
}$ equal to zero used before Corollary \ref{trivial} does not allow us to prove Theorem
\ref{Thm L=L=L}. Thus we extend the distance function, since this extension will have better properties.

\begin{proposition}\label{good}
Let $\Gamma$ be a leaf-manifold. Let $\mathcal U$ be given by Proposition \ref{def dist sign} and let $\mathcal V$ be an open neighborhood of $\Gamma$ such that $\overline{\mathcal V}\subset \mathcal U$. Then there exists a function $\phi:\mathbb R^N\rightarrow \mathbb R$, such that
\begin{description}
  \item[i)] $\phi(x)=d\left(\cdot,\Gamma\right)$ on $\overline{\mathcal{V}}$
  \item[ii)] the process $\phi\left(X_{t}\right)$ is a continuous semimartingale
  \item[iii)] for each compact ball $B\subset\mathbb R^N$ there exists $\varepsilon_1>0$ such that for each $\varepsilon<\varepsilon_1$ we have $\left\{x\in B:\ \phi(x)<\varepsilon\right\}=\left\{x\in B:\ d(x,\Gamma)<\varepsilon\right\}.$
\end{description}
Such a function $\phi$ will be called good extension of $d\left(\cdot,\Gamma\right)$.
\end{proposition}
Proof in Appendix B.

\begin{lemma}\label{newlemma}
Let $X$ be a continuous semimartingale in $\mathbb{R}^{N}$, $\phi$ a good extension of $d(\cdot,\Gamma)$ and $\mathcal{V}$ an open neighborhood of $\Gamma$ satisfying Proposition \ref{good}. Then one has:
\begin{description}
\item[iv)] for each $t$ such that $X_t\in \mathcal{V}$ we have that
$$d\langle \phi(X)\rangle_t=\sum_{i,j=1}^{N}\partial_{i}\delta_\Gamma\left(  X_{t}\right)  \partial_{j}%
\delta_\Gamma\left(  X_{t}\right)  d\left\langle X^{i},X^{j}\right\rangle _{t}.$$
Moreover if $X$ is a Brownian semimartingale, $d\langle \phi(X)\rangle_t=dt$.
\item[v)] If $X$ is a Brownian semimartingale, then $\phi$ controls $X$ in quadratic variation on $\mathcal{V}$.
\end{description}
\end{lemma}
Proof in Appendix B.

\subsection{The density $\mathcal L_{A,\phi}^a$ when $X$ is a Brownian semimartingale and $\phi=\delta_\Gamma$\label{Sect distance with sign}}
Here we will restrict ourselves to Brownian semimartingales, since they
mix well with the properties of $d(\cdot, \Gamma)$. Thus $\mathcal{L}_{i,A,\phi}%
^{a}$ is independent of $i$ and will be denoted by $\mathcal{L}_{A,\phi}^{a}$. With this choice we may identify $\mathcal{L}_{i,A,\phi}^{a}$
as the the local time of the 1D semimartingale $\phi\left(  X\right)  $ and
deduce the existence of a c\`{a}dl\`{a}g version. Denote by $L_{T}^{a}\left(  Y\right)$ the c\`{a}dl\`{a}g modification of the local time at $a$ on $\left[
0,T\right]  $ of a continuous semimartingale $Y$ (see \cite{RevuzYor}, Theorem 1.7, Chapter VI).

\begin{definition}
Let $Y$ be a continuous semimartingale in $\mathbb R$ and $a\in \mathbb R$, then we define the symmetrical local time of $Y$ in $a$ as
$$\tilde L^a_T(Y):=\lim_{\varepsilon\rightarrow0}\frac{1}{2\varepsilon}\int_{0}^{T}%
1_{(a-\varepsilon,a+\varepsilon)}(Y_t)d\langle Y\rangle_t.$$
\end{definition}

For each fixed $\omega \in\Omega$, it coincides with $L^a_T(Y)(\omega)$ except if $a$ is a point of discontinuity for it, and in that case
$$\tilde L^a_T(Y)(\omega)=\frac{L^a_T(Y)(\omega)+ L^{a-}_T(Y)(\omega)}{2}$$
where $L^{a-}_T(Y)(\omega)$ is the left limit of the local time in $a$ (see \cite{RevuzYor}, Chapter VI). In general $\tilde L^a_T(Y)$ is a modification of $L^a_T(Y)$, and whenever the local time is continuous they coincide.

\begin{theorem1}
\label{Thm L=L=L}Let $X$ be a Brownian semimartingale in $\mathbb{R}^{N}$. Let $\Gamma$ be a leaf-manifold and $\mathcal U$ be an open neighborhood of $\Gamma$ as in Proposition \ref{def dist sign}. Let $\phi:\mathbb{R}^{N}\rightarrow\mathbb{R}$ be a good
extension of $d(\cdot,\Gamma)$ and $\mathcal{V}\subset \mathcal U$ be an open neighborhood of $\Gamma$ satisfying Proposition \ref{good}. Then the assumptions of Theorem
\ref{Thm occupation time formula localized} are satisfied with $A=\mathcal{V}$.  Moreover for each fixed $\omega\in\Omega$ there exists $\varepsilon_1(\omega)>0$ such that%
\begin{equation}
\mathcal{L}_{A,\phi}^{a}=\tilde L_{T}^{a}\left(  \phi\left(  X\right)  \right)
\label{main claim}%
\end{equation}
for a.e. $a<\varepsilon_1(\omega)$, and they are both null if $a<0$.

In particular, on the random interval $(-\infty,\varepsilon_1]$ the process $(\omega,a)\mapsto \mathcal L^a_{A,\phi}(\omega)$ is the modification of a c\`{a}dl\`{a}g process.
\end{theorem1}

\begin{proof}
Using Lemma \ref{newlemma}, part (v), we have that the assumptions of Theorem \ref{Thm occupation time formula localized}
are satisfied. Let us prove (\ref{main claim}). Using part (iv) of the same Corollary we have%
$$d\left\langle \phi\left(  X\right)  \right\rangle _{t}=dt$$
because $X$ is a Brownian semimartingale, and $A=\mathcal{V}$. Hence for a.e. $a$ (we use the formula for $\mathcal{L}_{A,\phi}^{a}$ given by
Theorem \ref{Thm occupation time formula localized})
\begin{align*}
\mathcal{L}_{A,\phi}^{a}  &  =\lim_{\varepsilon\rightarrow0}\frac{1}{2\varepsilon
}\int_{0}^{T}1_{X_{t}\in A}1_{(a-\varepsilon,a+\varepsilon)}(\phi(X_{t}))dt\\
&  =\lim_{\varepsilon\rightarrow0}\frac{1}{2\varepsilon}\int_{0}^{T}%
1_{X_{t}\in A}1_{(a-\varepsilon,a+\varepsilon)}(\phi(X_{t}))d\left\langle
\phi\left(  X\right)  \right\rangle _{t}.
\end{align*}
For each fixed $\omega\in\Omega$ the trajectory of $X_t$, $t\in[0,T]$ remains inside a compact ball $B(\omega)$; then we can use Proposition \ref{good}, part (iii) and obtain that there exists $\varepsilon_1(\omega)>0$ such that for each $\varepsilon<\varepsilon_1(\omega)$ we have $\left\{x\in B:\ \phi(x)<\varepsilon_1\right\}=\left\{x\in B:\ d(x,\Gamma)<\varepsilon_1\right\}$. In particular if $a<\varepsilon_1(\omega)$ and $\varepsilon<|a-\varepsilon_1(\omega)|$ then
\begin{align}\label{identityA}
1_{X_{t}\in A}1_{(a-\varepsilon,a+\varepsilon)}(\phi(X_{t}))=1_{(a-\varepsilon
,a+\varepsilon)}(\phi(X_{t}))
\end{align}
hence, for a.e. $a<\varepsilon_1(\omega)$,
\begin{align*}
\mathcal{L}_{A,\phi}^{a}  &  =\lim_{\varepsilon\rightarrow0}\frac{1}{2\varepsilon
}\int_{0}^{T}1_{(a-\varepsilon,a+\varepsilon)}(\phi(X_{t}))d\left\langle
\phi\left(  X\right)  \right\rangle _{t}=\tilde L_{T}^{a}\left(  \phi\left(  X\right)  \right)  ,
\end{align*}
that is a modification of the c\`{a}dl\`{a}g process $L_{T}^{a}\left(  \phi\left(  X\right)  \right)$.
The proof is complete.
\end{proof}

\begin{corollary}
Under the hypotheses of the previous theorem, if there exists $\varepsilon>0$ such that $\mathcal U=\mathcal U_\varepsilon(\Gamma)$, we can take $\varepsilon_1(\omega)=\varepsilon_2<\frac{\varepsilon}{2}$ for each $\omega\in\Omega$. Hence we have that $\mathcal{L}_{A,\phi}^{a}$ is the modification of a c\`{a}dl\`{a}g process for a.e$.$ $a<\varepsilon_2$.
\end{corollary}

\begin{proof}
Repeat the proof of Proposition \ref{good}, part (iii), taking $B=\mathbb R^N$ and then the proof of the previous theorem.
\end{proof}

\subsection{Manifolds with singularities}
Until now in this section we have solved (in two ways, namely with the signed
distance and the distance function) the problem of the construction of a
suitable function $\phi$, given a $C^{2}$ manifold $\Gamma$ (with suitable
additional properties). To show that, potentially, the theory developed in
this paper may adapt to manifolds with singularities, we give here two
examples of construction of $\phi$ when the set $\Gamma$ is less regular:
first a manifold with some Lipschitz point, then the transversal union of
smooth manifolds. Writing general statements in such cases
turns out to be particularly annoying; thus we limit ourselves to show some
particular examples, that the reader will easily conceptualize.

\begin{example}
Let $D\subset\mathbb{R}^{2}$ be the square
$$D=\left[  0,1\right]^{2}=\left\{  x=\left(  x_{1},x_{2}\right)  :x_{1}\in\left[  0,1\right],x_{2}\in\left[  0,1\right]  \right\}.$$
Let $\Gamma=\partial D$. It is a
piecewise smooth manifold. The function $\delta_{\Gamma}$ defined in Example \ref{Example boundary} is
smooth except on the following sets:%
\begin{align*}
& \left\{  x_{1}\in\left[  0,1\right]  ,x_{2}=x_{1}\right\},\quad \left\{  x_{1}\in\left[  0,1\right]  ,x_{2}=1-x_{1}\right\}
\end{align*}
where it is continuous with side derivatives.
We do not have the properties of Proposition \ref{def dist sign} but still we may check directly the properties of Definition
\ref{definition non-degenerate}, for instance in the case when $X$ is a Brownian motion. Indeed we have that
$$|x_1|<|x_2| \Leftrightarrow (x_1^+<x_2^+) \vee ((-x_1)^+<(-x_2)^+).$$
Then we consider
$$x_1^+-\left(x_1^+-x_2^+\right)^+=\left\{
                                     \begin{array}{ll}
                                       x_1^+, & \hbox{if } x_1^+<x_2^+ \\
                                       x_2^+, & \hbox{if } x_1^+\ge x_2^+
                                     \end{array}
                                   \right.$$
and
$$(-x_1)^+-\left((-x_1)^+-(-x_2)^+\right)^+=\left\{
                                     \begin{array}{ll}
                                       (-x_1)^+, & \hbox{if } (-x_1)^+<(-x_2)^+ \\
                                       (-x_2)^+, & \hbox{if } (-x_1)^+\ge (-x_2)^+
                                     \end{array}
                                   \right.$$
So we obtain
$$\phi(x)=x_1^+-\left(x_1^+-x_2^+\right)^+-\left((-x_1)^+-\left((-x_1)^+-(-x_2)^+\right)^+\right).$$
In particular we can use It\^o-Tanaka theorem on the single functions that compose $\phi$: the process $\phi(X)$ is a semimartingale with quadratic variation equal a.s$.$ to
$$\sum_{ij}\partial_i\phi(X_t)\partial_j\phi(X_t)d\langle X^i,X^j\rangle_t=dt.$$
Thus we could apply Proposition \ref{Prop suff cond for control} and then Theorem \ref{Thm occupation time formula localized}.
\end{example}

\begin{example}
Let $\Gamma\subset\mathbb{R}^{2}$ be the union of the two lines:
\[
\Gamma=\left\{  x:x_{2}=x_{1}\right\}  \cup\left\{  x:x_{2}=-x_{1}\right\}  .
\]
We introduce the sets%
\begin{align*}
D_{1}=\left\{  x:x_{1}>0,-x_{1}<x_{2}<x_{1}\right\},\quad D_{2}=\left\{  x:x_{2}>0,-x_{2}<x_{1}<x_{2}\right\}
\end{align*}
$D_{3}=-D_{1}$, $D_{4}=-D_{2}$. We set%
\[
\phi\left(  x\right)  =\left\{
\begin{array}
[c]{ccc}%
d\left(  x,\Gamma\right)   & \text{if} & x\in D_{1}\cup D_{3}\\
-d\left(  x,\Gamma\right)   & \text{if} & x\in D_{2}\cup D_{4}%
\end{array}
\right.
\]
and $\phi\left(  x\right)  =0$ on $\Gamma$. It preserves some properties of
the signed distance function. It is Lipschitz continuous everywhere, but it is
not differentiable on the axes%
\[
\left\{  x:x_{1}=0\right\}  \cup\left\{  x:x_{2}=0\right\}  .
\]
If $X$ is a Brownian motion, Definition \ref{definition non-degenerate} applies and thus Proposition \ref{Prop suff cond for control} and Theorem \ref{Thm occupation time formula localized} hold. Indeed we consider that
$$x_1^+-\left(x_1^+-x_2^+\right)^+=\left\{
                                     \begin{array}{ll}
                                       x_1^+, & \hbox{if } x_1^+<x_2^+ \\
                                       x_2^+, & \hbox{if } x_1^+\ge x_2^+
                                     \end{array}
                                   \right.$$
Then we define
$$\psi(x)=-\left|\left((-x_1)^+,(-x_2)^+\right)\right|$$
and so
$$\phi(x)=x_1^+-\left(x_1^+-x_2^+\right)^++\psi(x).$$
We can apply again the It\^o-Tanaka theorem to all the single functions (remind that the modulus is convex) and get the same result
as in the previous example.
\end{example}

\section{Local times with respect to a codimension-1 manifold\label{2}}
In this section we shall introduce the notion of local time at an $(N-1)$-dimensional
manifold $\Gamma$, on $\left[0,T\right]$, of a continuous semimartingale
$X$ in $\mathbb{R}^{N}$;\ it will be denoted by $L_{T}^{\Gamma}\left(
X\right)  $. Then we collect here its relation with $\mathcal{L}_{A,\phi}^{a}$ and, in the special case that $\Gamma$ is globally a graph, with the notion used in \cite{Peskir}.

We assume that $X$ is a continuous semimartingale in
$\mathbb{R}^{N}$, defined on a filtered probability space $\left(\Omega,\mathcal{A},\mathcal{F}_{t},P\right)$.

\begin{theorem}
\label{Thm local time on manifolds} Let $\Gamma$ be a leaf-manifold. Let $\mathcal U$ be an open neighborhood of $\Gamma$ satisfying Proposition \ref{def dist sign}. Then the limit%
\[
L_{T}^{\Gamma}\left(  X\right)  :=\lim_{\varepsilon\rightarrow0}\frac
{1}{2\varepsilon}\sum_{i,j=1}^{N}\int_{0}^{T}1_{[0,\varepsilon)}\left(  d\left(  X_{t}%
,\Gamma\right)  \right)  \partial_{i}\delta_\Gamma\left(
X_{t}\right)  \partial_{j}\delta_\Gamma\left(  X_{t}\right)  d\left\langle
X^{i},X^{j}\right\rangle _{t}
\]
is well defined and a.s$.$ exists. It will be called geometric local time of $X$ at $\Gamma$ on $\left[  0,T\right]  $. If $X$ is a Brownian semimartingale, then%
\[
L_{T}^{\Gamma}\left(  X\right)  =\lim_{\varepsilon\rightarrow0}\frac
{1}{2\varepsilon}\int_{0}^{T}1_{[0,\varepsilon)}\left(  d\left(  X_{t}%
,\Gamma\right)  \right)  dt.
\]
\end{theorem}

\begin{remark}
Let $\psi$ be an arbitrary $C^2(\mathbb R^N)$ extension of the signed distance $\delta_\Gamma$. Then to be rigorous we should write $\partial_{i}\psi$ instead of $\partial_{i}\delta_\Gamma$, because $\delta_\Gamma$ is not defined outside $\mathcal U$. But the limit does not depend on the choice of the extension: indeed for each fixed $\omega\in\Omega$ the trajectory of the process $X$ remains inside a compact ball $B$, and so we may take an $\varepsilon_0(\omega)>0$ such that $\mathcal U\cap B\supseteq \mathcal U_{\varepsilon_0}(\Gamma)\cap B$. Hence for each $\varepsilon<\varepsilon_0(\omega)$ the limit depends only on $\delta_\Gamma$.
\end{remark}

\begin{proof}
Let $\phi$ be the good extension of $d(\cdot,\Gamma)$ defined in Proposition \ref{good}, by L\'{e}vy characterization of local times (see \cite{RevuzYor}, Corollary 1.9, Chapter VI) there exists a.s.
$$I:=\lim_{\varepsilon\rightarrow0}\frac{1}{\varepsilon}\int_{0}^{T}1_{[0,\varepsilon)}\left(  \phi\left(  X_{t}\right)
 \right)  d\left\langle  \phi\left(  X\right)\right\rangle _{t}.$$
The first claim is proved because for each fixed $\omega\in\Omega$ we have $L_{T}^{\Gamma}\left(  X\right)  =\frac{1}{2}I$. Indeed the trajectory $X_t(\omega)$, for each $t\in[0,T]$ remains inside a compact ball $B(\omega)$; then we can use Proposition \ref{good}, part (iii) and obtain that there exists $\varepsilon_1(\omega)>0$ such that for each $\varepsilon<\varepsilon_1(\omega)$ we have $\left\{x\in B:\ \phi(x)<\varepsilon_1\right\}=\left\{x\in B:\ d(x,\Gamma)<\varepsilon_1\right\}$; and in particular for each $\varepsilon<\varepsilon_1$ we have $1_{[0,\varepsilon)}\left(  \phi\left(  X_{t}\right)\right)=1_{[0,\varepsilon)}\left(  d\left(  X_{t},\Gamma\right)  \right)$. Moreover, following Lemma \ref{newlemma}, part (iv), for each $t$ such that $X_t\in \mathcal{V}$ we have that
$$\langle \phi(X)\rangle_t=\sum_{i,j=1}^{N}\partial_{i}\delta_\Gamma\left(  X_{t}\right)  \partial_{j}%
\delta_\Gamma\left(  X_{t}\right)  d\left\langle X^{i},X^{j}\right\rangle _{t}$$
and this is true whenever $d\left(  X_{t},\Gamma\right)<\varepsilon_1$. In particular, thanks to the same Lemma, if $X$ is a Brownian semimartingale we obtain also the second claim.
\end{proof}

\begin{remark}
The formula given above which defines $L_{T}^{\Gamma}\left(  X\right)  $ in
the general case may look strange at first sight. However, it is the natural
one if we think to the particular case of an hyperplane $\Gamma$. In that
case, the natural definition would be the classical local time (which includes
the time-change due to the quadratic variation) of the projection of $X$ along
the normal to $\Gamma$. This is the formula above, as we also show below in
Proposition \ref{Proposition example}.
\end{remark}

Here we will suppose that $X$ is a Brownian semimartingale and we will also assume that there exists an $\varepsilon>0$ such that $\mathcal U=\mathcal U_\varepsilon(\Gamma)$: this is because we need that for each $a<\varepsilon$ the level sets $\Gamma_a=\{x:\ d(x,\Gamma)=a\}$ are leaves-manifolds.

The following geometric lemma is about the relation between $d(\cdot,\Gamma)$ and $d(,\cdot,\Gamma_a)$.
\begin{lemma}\label{orto1}
Let be $a>0$ and $\varepsilon>0$, then the following properties are equivalent:

i) $d\left(  x,\Gamma_{a}\right)  \in\lbrack0,\varepsilon)$

ii) $d(x,\Gamma)  \in\left(a-\varepsilon,a+\varepsilon\right)  $.
\end{lemma}
Proof in Appendix A.

\begin{corollary}\label{orto}
If $\varepsilon_0>a>0$ and $\varepsilon<|a-\varepsilon_0|$, then the following properties are equivalent:

i) $x\in\mathcal U_{\varepsilon_0}(\Gamma)$, $d\left(  x,\Gamma_{a}\right)  \in\lbrack0,\varepsilon)$

ii)\ $x\in\mathcal U_{\varepsilon_0}(\Gamma)$, $d(x,\Gamma)  \in\left(a-\varepsilon,a+\varepsilon\right)  $.
\end{corollary}

\begin{proof}
We apply the previous Lemma considering that for each $x\in\mathcal U_{\varepsilon_0}(\Gamma)$, both the neighborhoods $\mathcal U_{\varepsilon}(\Gamma_a)$ and $\left\{x:\ d(x,\Gamma)  \in\left(a-\varepsilon,a+\varepsilon\right)\right\}$ are subsets of $\mathcal U_{\varepsilon_0}(\Gamma)$.
\end{proof}

\begin{theorem1}\label{Thm L=L}
Let $X$ be a Brownian semimartingale in $\mathbb{R}^{N}$. Let $\Gamma$ be a leaf-manifold and $\mathcal U$ be an open neighborhood of $\Gamma$ as in Proposition \ref{def dist sign}. Let $\phi:\mathbb{R}^{N}\rightarrow\mathbb{R}$ be a good extension of $d(\cdot,\Gamma)$ and $\mathcal{V}\subset \mathcal U$ be an open neighborhood of $\Gamma$ satisfying Proposition \ref{good}. Then the assumptions of Theorem \ref{Thm occupation time formula localized} are satisfied with $A=\mathcal{V}$. Moreover suppose there exists $\varepsilon_0>0$ such that $\mathcal U=\mathcal{U}_{\varepsilon_{0}}(\Gamma)$ so that we can take also $\mathcal V=\mathcal{U}_{\varepsilon_{1}}(\Gamma)$ for a fixed $\varepsilon_1\in(0,\varepsilon_0)$. Then we have
\begin{equation*}
\mathcal{L}_{A,\phi}^{a}=L_{T}^{\Gamma_{a}}\left(  X\right)
\end{equation*}
for a.e. $a\in\left[ 0,\varepsilon_{0}\right)  $, where
$\Gamma_{a}=\left\{  x\in\mathcal V:d(x,\Gamma)=a\right\}$.
\end{theorem1}

\begin{proof}
The set $\mathcal V=\mathcal U_{\varepsilon_1}(\Gamma)$ has a tubular neighborhood structure (see \cite{Kosinski}) that we can define following Proposition \ref{posit}, indeed the regularity of $\delta_\Gamma$ implies that the manifold has positive reach. Hence for each $a<\varepsilon_1$, the set $\Gamma_a$ is a leaf-manifold. Thus we can apply Theorem \ref{Thm local time on manifolds} and define%
\[
L_{T}^{\Gamma_{a}}\left(  X\right)  :=\lim_{\varepsilon\rightarrow0}\frac
{1}{2\varepsilon}\int_{0}^{T}1_{[0,\varepsilon)}(d(X_{t},\Gamma_{a}))dt.
\]
From Corollary \ref{orto}, for each $a\in\left[
0,\varepsilon_{1}\right)  $ and very small $\varepsilon$ we
have%
\begin{align*}
1_{[0,\varepsilon)}(d(x,\Gamma_{a}))  &  =1_{x\in A}1_{(a-\varepsilon
,a+\varepsilon)}(d(x,\Gamma))=1_{x\in A}1_{(a-\varepsilon,a+\varepsilon)}(\phi(x)),
\end{align*}
because $A=\mathcal{V}$. Hence
\begin{align*}
L_{T}^{\Gamma_{a}}\left(  X\right)   &  =\lim_{\varepsilon\rightarrow0}%
\frac{1}{2\varepsilon}\int_{0}^{T}1_{x\in A}1_{(a-\varepsilon,a+\varepsilon
)}(\phi(x))dt=\mathcal{L}_{A,\phi}^{a}.
\end{align*}
\end{proof}

\subsection{The graph local time}
The geometric local time introduced above is intrinsic, in the sense that it
is independent of the coordinate system of $\mathbb{R}^{N}$:\ indeed, it is
defined only in terms of $\Gamma$, $X$, and the function $\delta_{\Gamma}$. In
this section we compare the geometric local times with the graph local times defined below and we show in
particular that they are different; moreover, $L_{X,T}^{\Gamma,graph}$ may
change value if we change the coordinate system used to describe $\Gamma$ as a
graph. We explain this by the simple example of Proposition \ref{Proposition example} below.

Let $g:\mathbb{R}^{N-1}\rightarrow\mathbb{R}$ be a $C^{2}$ function and let
$\Gamma$ be its graph:
\[
\Gamma=\left\{  \left(  x^{1},...,x^{N}\right)  \in\mathbb{R}^{N}%
:x^{N}=g\left(  x^{1},...,x^{N-1}\right)  \right\}  .
\]
Following \cite{Peskir}, let us define the \textit{graph local
time\footnote{This name is given here only to distinguish the notion from the
geometric local time given above.} of }$X$ at $\Gamma$ as
\[
L_{T}^{\Gamma,graph}\left(  X\right)  =\lim_{\varepsilon\rightarrow0}\frac
{1}{2\varepsilon}\int_{0}^{T}1_{[0,\varepsilon)}\left(  \left\vert
Y_{t}\right\vert \right)  d\left\langle Y\right\rangle _{t}%
\]
where
\[
Y_{t}=X_{t}^{N}-g\left(  X_{t}^{1},...,X_{t}^{N-1}\right)  .
\]

Given $v\in\mathbb{R}^{N}$, by $v\cdot X$ we mean the process $\sum_{i=1}^{N}v^{i}X^{i}$.

\begin{proposition}
\label{Proposition example}In $\mathbb{R}^{N}$, let $\Gamma$ be the $\left(
N-1\right)  $-dimensional subspace orthogonal to a given unitary vector $v$.
Then%
\[
L_{T}^{\Gamma}\left(  X\right)  =\lim_{\varepsilon\rightarrow0}\frac
{1}{2\varepsilon}\int_{0}^{T}1_{[0,\varepsilon)}\left(  d\left(  X_{t}%
,\Gamma\right)  \right)  d\left\langle v\cdot X\right\rangle _{t}%
\]
Given a system of coordinates in $\mathbb{R}^{N}$ (we write $x=\left(
x^{1},...,x^{N}\right)  $), let $a=\left(  a^{1},...,a^{N-1}\right)
\in\mathbb{R}^{N}$ be the vector such that $\Gamma$ is defined by the
equation $x^{N}=\sum_{i=1}^{N-1}a^{i}x^{i}$. Then%
\[
L_{T}^{\Gamma,graph}\left(  X\right)  =\sqrt{1+\left\vert a\right\vert ^{2}%
}L_{T}^{\Gamma}\left(  X\right)  .
\]

\end{proposition}

\begin{proof}
We choose the system of coordinates in $\mathbb{R}^{N}$ and the vector
$a\in\mathbb{R}^{N-1}$ as in the statement. One can introduce a global signed
distance function $\delta_{\Gamma}$, such that
\[
\nabla\delta_{\Gamma}\left(  x\right)  =v\qquad\text{for all }x\in
\mathbb{R}^{N}.
\]
Hence%
\begin{align*}
L_{T}^{\Gamma}\left(  X\right)   &  =\lim_{\varepsilon\rightarrow0}\frac
{1}{2\varepsilon}\sum_{i,j=1}^{N}\int_{0}^{T}1_{[0,\varepsilon)}\left(
d\left(  X_{t},\Gamma\right)  \right)  v^{i}v^{j}d\left\langle X^{i}%
,X^{j}\right\rangle _{t}\\
&  =\lim_{\varepsilon\rightarrow0}\frac{1}{2\varepsilon}\int_{0}%
^{T}1_{[0,\varepsilon)}\left(  d\left(  X_{t},\Gamma\right)  \right)
d\left\langle v\cdot X\right\rangle _{t}.
\end{align*}
The first relation is proved.

We also have, using the coordinates,
\[
v\cdot X=\frac{1}{\sqrt{1+\left\vert a\right\vert ^{2}}}\left(  X^{N}%
-a\cdot\overline{X}\right)  =\frac{1}{\sqrt{1+\left\vert a\right\vert ^{2}}%
}Y_{t}%
\]
where we have written $X=\left(  \overline{X},X^{N}\right)  $. Hence%
\[
L_{T}^{\Gamma}\left(  X\right)  =\frac{1}{1+\left\vert a\right\vert ^{2}}%
\lim_{\varepsilon\rightarrow0}\frac{1}{2\varepsilon}\int_{0}^{T}%
1_{[0,\varepsilon)}\left(  d\left(  X_{t},\Gamma\right)  \right)
d\left\langle Y_{t}\right\rangle _{t}.
\]

Finally, it is a simple exercise to check that $d\left(  X_{t},\Gamma\right)
=r$ if and only if $\left\vert Y_{t}\right\vert =r\sqrt{1+\left\vert
a\right\vert ^{2}}$, and thus $d\left(  X_{t},\Gamma\right)  \in
\lbrack0,\varepsilon)$ if and only if $\left\vert Y_{t}\right\vert
<\varepsilon\sqrt{1+\left\vert a\right\vert ^{2}}$.

\begin{align*}
L_{T}^{\Gamma}\left(  X\right)   &  =\frac{1}{1+\left\vert a\right\vert ^{2}%
}\lim_{\varepsilon\rightarrow0}\frac{1}{2\varepsilon}\int_{0}^{T}%
1_{[0,\varepsilon\sqrt{1+\left\vert a\right\vert ^{2}})}\left(  \left\vert
Y_{t}\right\vert \right)  d\left\langle Y\right\rangle _{t}\\
&  =\frac{1}{\sqrt{1+\left\vert a\right\vert ^{2}}}\lim_{\varepsilon
\rightarrow0}\frac{1}{2\varepsilon}\int_{0}^{T}1_{[0,\varepsilon)}\left(
\left\vert Y_{t}\right\vert \right)  d\left\langle Y\right\rangle _{t}=\frac{1}{\sqrt{1+\left\vert a\right\vert ^{2}}}L_{T}^{\Gamma,graph}\left(
X\right)  .
\end{align*}
The proof is complete.
\end{proof}

When $g$ is not linear, at present we can only guess that a suitable
localization argument could lead to the identity
\[
L_{T}^{\Gamma}\left(  X\right)  =\lim_{\varepsilon\rightarrow0}\frac
{1}{2\varepsilon}\int_{0}^{T}1_{[0,\varepsilon)}\left(  \left\vert
Y_{t}\right\vert \right)  \frac{1}{\sqrt{1+\left\vert \nabla g\left(
X_{t}^{1},...,X_{t}^{N-1}\right)  \right\vert ^{2}}}d\left\langle
Y\right\rangle _{t}%
\]
but the proof is not trivial and will not be discussed further here.

\section*{Appendix}
\subsection*{A)$\quad$ Some proofs of geometry}
\paragraph{Proof of Proposition \ref{def dist sign}.}
\begin{proof2}
For each point $x\in\Gamma$ there exists $\varepsilon_0>0$ and a diffeomorphism $\psi$ from the ball $B(x,\varepsilon_0)$ into $B(0,1)\subset\mathbb R^N$, such that the restriction of $\psi$ on $\Gamma\cap B(x,\varepsilon_0)$ is a diffeomorphism onto the $(N-1)$-dimensional disc. Then there exists $\varepsilon_1\in\left(0,\frac{\varepsilon_0}{2}\right)$ such that $\Gamma\cap \overline{B(x,\varepsilon_0)}$ because of its compactness has a tubular neighborhood of width equal to $\varepsilon_1$, on which the distance from the manifold coincides with the distance on the normal bundle (see \cite{Kosinski}). Moreover we can use the orientation of the normal bundle to locally define a function $\delta^x_\Gamma$ satisfying (i) and (ii) w.r.t$.$ $\Gamma\cap B(x,\varepsilon_0)$. If we restrict the previous tubular neighborhood around the submanifold $\Gamma\cap B(x,\varepsilon_1)$, and we call it $U_x$, we have also that $d(\cdot,\Gamma\cap B(x,\varepsilon_1))=d(\cdot,\Gamma)$ on $U_x$. Moreover we can define $\mathcal U=\bigcup_{x\in\Gamma} U_x$ and we obtain a global neighborhood of $\Gamma$. Then thanks to its orientability we have that all the local signed distances have a compatible signature: for each $x,y\in\Gamma$, $y\in U_x\cap U_z$ we have $\delta^x_\Gamma(y)=\delta^z_\Gamma(y)$. For each $y\in\mathcal U$ exists $x\in\Gamma$ such that $y\in U_x$ and we can define $\delta_\Gamma(y)=\delta^x_\Gamma(y)$; it satisfies both (i) and (ii).
\end{proof2}

\paragraph{Proof of Lemma \ref{prop dist}.}
\begin{proof2}
The functions $d\left(  \cdot,\Gamma\right)  $ and
$\delta_{\Gamma}\left(  \cdot\right)  $ are both continuous and different from
zero in the open set $\mathcal{U}\backslash\Gamma$. The ratio $\frac
{\delta_{\Gamma}\left(  x\right)  }{d\left(  x,\Gamma\right)  }$ is a
continuous well defined function on $\mathcal{U}\backslash\Gamma$, equal to
$\pm1$, hence it is constant on each connected component of $\mathcal{U}%
\backslash\Gamma$. Property (iii) is proved. Property (iv) is an easy
consequence of (i) and (iii). Property (v) is true by contradiction: if exists $x\in U$ such that $P_F(x)\supseteq\{y_1,y_2\}$ with $y_1\neq y_2$, then $\frac{x-y_1}{|x-y_1|}\neq\frac{x-y_2}{|x-y_2|}$. Then if we substitute $v_1= \frac{x-y_1}{|x-y_1|}$ in (\ref{ceco}) we have that $1=\frac{x-y_1}{|x-y_1|}\cdot\frac{x-y_1}{|x-y_1|}<\frac{x-y_1}{|x-y_1|}\cdot\frac{x-y_2}{|x-y_2|}$ by Cauchy-Schwartz inequality, hence $D_{v_1} d(x,F)=v_1\cdot\frac{x-y}{|x-y|}$. Vice versa if we substitute $v_2= \frac{x-y_2}{|x-y_2|}$ in (\ref{ceco}) we obtain $D_{v_2} d(x,F)=v_2\cdot\frac{x-y}{|x-y|}$. This is incompatible with the differentiability of $\delta_\Gamma$ in $U$. About property (vi), if $x\in U\backslash\Gamma$, thanks to (iii) we have
$$| \nabla\delta_{\Gamma}\left(  x\right)|=| \nabla d(\cdot,\Gamma)\left(  x\right)|=\max_{|v|=1} |D_v d(x,F)|=1$$
where we substituted $v=P_{\Gamma}\left(  x\right)$ in (\ref{ceco}) and $P_{\Gamma}(\cdot)$ is well defined thanks to (v). If otherwise $x\in\Gamma$, we have $|\nabla\delta_{\Gamma}\left(  x\right)|=1$ by continuity thanks to (i).
\end{proof2}

\paragraph{Proof of Proposition \ref{posit}.}
\begin{proof2}
We have that $\mathcal U_\varepsilon(\Gamma)$ is the union of all the normal segments $N_x$, $x\in\Gamma$, of lengths $\varepsilon$ on both sides of $\Gamma$, and this union must be disjoint. Indeed by contradiction let $x_1,x_2\in\Gamma$, $x_1\neq x_2$, be such that there exists $y\in N_{x_1}\cap N_{x_2}$, and suppose that $d(y,x_1)\ge d(y,x_2)$. Then the continuous function $d(\cdot,x_1)-d(\cdot,x_2)$ is positive in $y$ and negative in $x_1$: there exists a point $z\in N_{x_1}$ such that $d(z,x_1)-d(z,x_2)=0$. This is incompatible with the positive reach property. Hence $\mathcal U_\varepsilon(\Gamma)$ is a global tubular neighborhood of $\Gamma$ (see \cite{Kosinski} for the definition): the distance from the manifold coincides on that neighborhood with the distance on the normal bundle and we can use the orientation to locally define a function $\delta_\Gamma$ satisfying properties (i) and (ii) of the Proposition \ref{def dist sign}.
\end{proof2}

\paragraph{Proof of Lemma \ref{orto1}.}
\begin{proof2}
The case $a=0$ is trivial: let us suppose $a\neq 0$. Let be $x\in\mathbb R^N$ and $y\in P_{\Gamma_a}(x)$ (the set of metric projections of $x$ on $\Gamma_a$), so that we have $d(x,\Gamma_a)=d(x,y)$. Then define $z\in P_{\Gamma}(y)$ and we have $d(y,z)=a$. Then
$$d(x,\Gamma)\le d(x,z) \le d(x,y)+d(y,z)=a+d(x,y)$$
and if $d(x,\Gamma_a)<\varepsilon$, we have $d(x,y)<\varepsilon$ and $d(x,\Gamma)\le a+\varepsilon$. Moreover define $z_1\in P_{\Gamma}(x)$, such that $d(x,\Gamma)=d(x,z_1)$. Then
$$d(x,\Gamma)= d(x,z_1) \ge d(y,z_1)-d(y,x)= a-d(x,y)$$
and if $d(x,\Gamma_a)<\varepsilon$, we have $d(x,y)<\varepsilon$ and $d(x,\Gamma)\ge a-\varepsilon$. We proved
$$a-\varepsilon\le d(x,\Gamma)(x)\le a+\varepsilon.$$
Vice versa define $y_1$ as the intersection of the segment linking $x$ and $z_1$. So we obtain $d(x,z_1)=d(x,y_1)+d(y_1,z_1)$ but $d(x,y_1)\ge d(x,\Gamma_a)$ and $d(y_1,z_1)\ge a$, so
$$d(x,\Gamma)=d(x,z_1)\ge d(x,\Gamma_a)+a.$$
In particular from (ii) we have $0\le d(x,\Gamma)\le a+\varepsilon$ and so $d(x,\Gamma_a)\le\varepsilon$.
\end{proof2}

\subsection*{B)$\quad$ The good extension construction}
\paragraph{Proof of Proposition \ref{good}.}
\begin{proof2}
Let $\theta:\mathbb R^N\rightarrow \mathbb R$ be a cut-off function of class $C^2$ such that $\theta(x)=1$ if $x\in\overline{ \mathcal{V}}$ and $\theta(x)=0$ if $x\in \mathcal U^c$. Thanks to Lemma \ref{prop dist}, the distance $d(\cdot,\Gamma)$ is regular in $\mathcal U\backslash\Gamma$, so we can regularize it on all $\mathbb R^N\backslash\Gamma$ without changing its value inside $\mathcal{V}$. We call $\tilde d$ this mollified distance. Then we take, for all $x\in\mathbb R^N$
$$\phi(x):=d_1(x)+d_2(x)$$
where $d_1(x):=\left(1-\theta(x)\right)\tilde d(x)$, $d_2(x):=\theta(x)d(x,\Gamma)$. We have that $d_1(X)$ is a continuous semimartingale because it is the composition of $X$ with a $C^2$ function. Moreover if we define $f(x)=\theta(x)\delta_\Gamma(x)$ for all $x\in \mathcal U$ and $f(x)=0$ when $x\in \mathcal U^c$, we obtain that $f$ is a function of class $C^2(\mathbb R^N)$, and $d_2(x)=|f(x)|$; so, thanks to the It\^o-Tanaka formula, also $d_2(X)$ is a continuous semimartingale, and $\phi(X)$ too.

To verify the third property, we remind that by compactness there exists an $\varepsilon_0>0$ such that $\mathcal{V}\cap B\supseteq \mathcal U_{\varepsilon_0}(\Gamma)\cap B$, so if we had taken mollifiers that do not change too much the value of $d(\cdot, \Gamma)$, there exists $\varepsilon_1<\varepsilon_0$ such that if $\phi(x)<\varepsilon<\varepsilon_1$ then $d(x,\Gamma)<\varepsilon_0$, and $x\in\mathcal U_{\varepsilon_0}(\Gamma)$. In particular if $x\in B$ then $x\in\mathcal{V}\cap B$ and $\phi(x)=d(x,\Gamma)$. Vice versa if $d(x,\Gamma)<\varepsilon<\varepsilon_1$ then again $d(x,\Gamma)<\varepsilon_0$ and we obtain like above that if $x\in B$ then $\phi(x)=d(x,\Gamma)$.
\end{proof2}

\paragraph{Proof of Lemma \ref{newlemma}.}
\begin{proof2}
Let $\psi$ be an arbitrary $C^2(\mathbb R^N)$ extension of $\delta_\Gamma$ from $\overline{\mathcal{V}}$. If $X_t\in \mathcal{V}$, hence $\phi(X_t)=|\delta_\Gamma(X_t)|=|\psi(X_t)|$, and we can apply It\^o-Tanaka formula to obtain
\begin{align*}
d\left\langle  |\psi\left(  X\right)|   \right\rangle _{t}
&  =\sum_{i,j=1}^{N}\partial_{i}\psi\left(  X_{t}\right)  \partial_{j}%
\psi\left(  X_{t}\right)  d\left\langle M^{i},M^{j}\right\rangle _{t}\\
&  =\sum_{i,j=1}^{N}\partial_{i}\delta_\Gamma\left(  X_{t}\right)  \partial_{j}%
\delta_\Gamma\left(  X_{t}\right)  d\left\langle X^{i},X^{j}\right\rangle _{t},
\end{align*}
where $M$ was the local martingale part of $X$. Then, if $X$ is a Brownian semimartingale, we have
$$\sum_{i,j=1}^{N}\partial_{i}\delta_\Gamma\left(  X_{t}\right)  \partial_{j}%
\delta_\Gamma\left(  X_{t}\right)  d\left\langle X^{i},X^{j}\right\rangle _{t}=|\nabla \delta_\Gamma(X_t)|^2dt=dt$$
by Lemma \ref{prop dist}. The proof of property (iv) is complete.

Moreover, if $X$ is a Brownian semimartingale, then $\phi$ satisfies the hypotheses of Definition \ref{definition non-degenerate} with $A=\mathcal{V}$ and $D_\phi=\mathcal{V}\backslash \Gamma$. Hence thanks to Proposition \ref{Prop suff cond for control}, it controls $X$ in quadratic variation on $\mathcal{V}$. Indeed $d\langle X\rangle_t=dt$ is trivially Lipshitz continuous, $\phi$ is locally 1-Lipshitz in $\mathcal{V}$ because it is equal to $d(\cdot,\Gamma)$ and we can use (\ref{ceco}) to estimate its partial derivatives. The property that
$$\left\{\omega\in\Omega:\ X_t(\omega)\in \Gamma\right\}=\left\{\omega\in\Omega:\ \phi(X_t(\omega))=0\right\}$$
implies that for a.e$.$ $t\in[0,T]$ we have
$$P\left\{\omega\in\Omega:\ X_t(\omega)\in \Gamma\right\}=P\left\{\omega\in\Omega:\ \phi(X_t(\omega))=0\right\}=0.$$
Otherwise we would have that on a not negligible event, there would exists $\Theta(\omega)\subset[0,T]$ with Lebesgue measure $\lambda>0$, such that $\forall t\in \Theta$ it would be $\phi(X_t)=0$. And this would contradict the Occupation time formula (see \cite{RevuzYor}, Chapter VI, Corollary 1.6)
$$\int_\Theta d\langle \phi(X)\rangle_s\le\int_0^T 1_{\{0\}}(\phi(X_s)d\langle \phi(X)\rangle_s=\int_{-\infty}^{+\infty} 1_{\{0\}}(a)da=0.$$
Indeed $\forall t\in\Theta$, $X_t\in\Gamma\subset \mathcal{V}$ and then we can use part (iv) of this lemma to obtain $d\langle \phi(X)\rangle_t=dt$, hence
$$\int_\Theta d\langle \phi(X)\rangle_s=\int_\Theta ds=\lambda>0.$$
Additionally $\phi(X)$ is trivially non-degenerate (following Definition \ref{definition non-degenerate}).
\end{proof2}

\end{document}